\documentclass[12pt, letterpaper]{article}
\usepackage{amssymb, amsmath}
\usepackage{amsthm}
\usepackage{quiver}
\usepackage{enumitem}

\title{Hasse-Schmidt derivations and locally trivial deformations in positive characteristic}
\author{Takuya Miyamoto\thanks{\textsc{Mathematical Sciences, the University of Tokyo, Meguro Komaba 3-8-1, Tokyo, Japan}\\
    \textit{E-mail address}: \texttt{miyamoto-takuya@g.ecc.u-tokyo.ac.jp}}}
\begin{document}
\maketitle

\begin{abstract}
    In this paper, we first study the lifting problem of Hasse–Schmidt derivations and then apply the results to the theory of locally trivial deformations of algebraic schemes in positive characteristic. As an application, we construct an algebraic curve whose locally trivial deformation functor does not satisfy Schlessinger's condition $(H_1)$.
\end{abstract}

\tableofcontents

\newtheorem{theorem}{Theorem}[section]
\newtheorem{proposition}[theorem]{Proposition}
\newtheorem{corollary}[theorem]{Corollary}
\newtheorem{lemma}[theorem]{Lemma}
\theoremstyle{definition}
\newtheorem{definition}[theorem]{Definition}

\newtheorem{assumption}[theorem]{Assumption}
\newtheorem{deflemma}[section]{Definition-Lemma}
\theoremstyle{remark}
\newtheorem{remark}{Remark}[theorem]
\newtheorem{example}{Example}[theorem]
\theoremstyle{definition}

\section{Introduction}
Throughout this paper, we work over an algebraically closed field $k$ of positive characteristic $p$. We assume that any $k$-algebra is unital and commutative.

In this paper, we study the following two lifting problems in positive characteristic: Suppose $$0\to k\cdot t\to B\to A\to 0$$ is a small extension of artinian local $k$-algebras of finite type. First, we study when a deformation automorphism of the trivial deformation $R\otimes_k A$ lifts to $R\otimes_k B$, where $R$ is a $k$-algebra of finite type. We apply the theory of Hasse-Schmidt derivations introduced in \cite{Has} to this problem. Second, let $X_0$ be an algebraic scheme and $X$ be a locally trivial deformation of $X_0$ over $A$ (a formal deformation that is Zariski-locally isomorphic to the trivial deformation). We study the obstruction structure of $X$ along $B\to A$. In other words, we use cohomological classes to give a characterization when $X$ lifts to a locally trivial deformation over $B$. (We note that if $X_0$ is smooth, then any formal deformation is locally trivial. Thus, the notion of locally trivial deformations coincides with usual deformations.)

Let us note that the theory of locally trivial deformation in positive characteristic is not well-established compared with characteristic $0$ case. For example, in characteristic $0$, the locally trivial deformation functor $\operatorname{Def}_{X_0}'$ satisfies Schlessinger's conditions (see \cite[Theorem 2.11]{Sch}, \cite[Remark 2.7]{BGL}), which is a fundamental and important fact. Then, as pointed out in \cite[p. 630]{flenner1987locally}, the natural question arises: What if $X$ is an algebraic scheme over a field of positive characteristic $p$? The proof for the characteristic $0$ case does not apply because it involves exponentials and in particular the fraction $\frac{1}{p!}$. In \cite[Theorem 2.4.1]{Ser}, it is asserted that locally trivial deformation functors satisfy $(H_1)$ even in positive characteristic, but there seems to be a subtle gap in the proof. We answer this question in the negative by the following theorem:
\begin{theorem}[= Theorem \ref{counterexample theorem}]\label{Main thm 2}
    For any algebraically closed base field of any positive characteristic, there exists a singular algebraic curve $C_0$ whose locally trivial deformation functor $\operatorname{Def}'_{C_0}$ does not satisfy Schlessinger's condition $(H_1)$.
\end{theorem}
This is our main result in this paper. As far as we have searched, such an example seems to be constructed for the first time ever. We cannot expect such an example to be smooth because if it were, its locally trivial deformation functor coincide with the usual deformation functor.

Another aim of this paper is to show that, even in positive characteristic, their behavior is tame to some extent and that there is room for further research in this area. For example, another important consequence is the following:

\begin{theorem}[= Corollary \ref{affine cor}]\label{Main thm}
    Every locally trivial deformation of an affine algebraic scheme $X_0$ over $k$ is trivial.
\end{theorem}

Note that the same result is an immediate consequence of \cite[Theorem 2.11]{Sch} in characteristic $0$ or if $X_0$ is smooth.


In Section 2, we study the Hasse-Schmidt derivations $\operatorname{HS}_k^m(R)$, where $R$ is a $k$-algebra essentially of finite type. This is equivalent to the notion of deformation automorphisms of $R\otimes_k (k[\varepsilon]/\varepsilon^{m+1})$. See \cite{Has}, \cite[Chapter 1]{Mac} and \cite[Section 27]{Mat} for basic facts of Hasse-Schmidt derivations. The problem of extending a Hasse-Schmidt derivation corresponds to extending a deformation automorphism along $k[\varepsilon]/\varepsilon^{m+1}\to k[\varepsilon]/\varepsilon^{m}$. To describe this problem, we introduce the $m$-th \emph{obstruction module} $\operatorname{Ob}^m_R$ whose definition is very natural (Definition \ref{obstruction sp M,Ob}, \ref{Ob}). Its definition relies on the first cotangent module $T^1_{R/k}$ (see \cite[Definition 1.1.6]{Ser} for example). The problem is that, by definition, it is merely an abelian group; it is not evident whether it is an $R$-module or even constitutes a sheaf over $\operatorname{Spec}R$. Only for $m=p^i$, we know that it has a natural module structure. To deal with this problem, we closely follow the methods of Hernández \cite[Lemma 2.10]{Her} to show that, for general $m$, we have $\operatorname{Ob}^m_R=\operatorname{Ob}^{p^i}_R$ where $i$ is the largest integer with $p^i|m$. In addition, we utilize the methods of Macarro \cite[Proposition 1.3.5]{Mac} to show that $\operatorname{Ob}^{p^i}_R$ is compatible with localization. Thus, we conclude that $\operatorname{Ob}^{p^i}_R$ constitutes a coherent sheaf (Proposition \ref{Ob module}, Theorem \ref{Ob sheaf}). Moreover, this demonstrates that Hasse-Schmidt derivations are useful not only in commutative algebra, but also in algebraic geometry. We also examine the relationship with m-integrable derivations \cite[Definition 2.1.1]{Mac}, which have been shown by Macarro to form coherent sheaves \cite[Corollary 2.3.7]{Mac}. In this section, we essentially use the fact that $k$ contains $n$-th roots of unity. Note that the definition of the $m$-th obstruction modules $\operatorname{Ob}^m_R$ does not use the characteristic $p$ assumption. However, since any Hasse-Schmidt derivation extends in characteristic $0$, we see that $\operatorname{Ob}^m_R$ is interesting only in positive characteristics.

In Section 3, we generalize the results of Section 2 to arbitrary small extensions. To keep the notation consistent with that of Section 2, we denote the group of deformation automorphisms of $R\otimes_k A$ by $\operatorname{HS}_k^A(R)$. For an arbitrary small extension $\Phi :0\to k\cdot t\to B\to A\to 0$, we can define the obstruction module $\operatorname{Ob}^\Phi_R$ in the same way. Again, the problem is that it is only an abelian group. To overcome this issue, we prove the following key result:
\begin{theorem}[= Theorem \ref{decomposition thm}]
For any $D\in \operatorname{HS}_k^A(R)$, we have a decomposition 
$$D=D^1\langle y_1 \rangle\circ \dots \circ D^m\langle y_m \rangle$$
where $y_1,\dots,y_m\in \mathfrak{m}_{A}$ and $D^i\in \operatorname{HS}_k^{r(A,y_i)}(R)$. (The symbols $D^i\langle y_i \rangle$, $r(A,y_i)$ will be defined in Definitions \ref{r def}, \ref{bracket def}).
 \end{theorem}
From this theorem, we can reduce the problem to the case of usual Hasse-Schmidt derivations. As a consequence, we deduce $\operatorname{Ob}^\Phi_R=\operatorname{Ob}^{p^i}_R$ where $i$ is an integer depending on $\Phi $. The proof of the above theorem is the most technical part of this paper. Its proof begins with defining artinian rings denoted by $A(e,a)$, where $e=(e_1,\dots,e_n)$ is a sequence of positive integers and $a$ a sequence in $k$. It is very close to the ring $A(e):=k[X_1,\dots,X_n]/(X_1^{e_1},\dots, X_n^{e_n})$. We first prove the theorem in the case $A=A(e,a)$ using geometric interpretation and various results in Section 2. Then, we reduce the general cases to $A(e,a)$ and this completes the proof. In this section, we use the perfectness of $k$ (we take $p^i$-th roots).

In Sections 4 and 5, we apply the previous results to locally trivial deformations of an algebraic scheme $X_0$ along a small extension $\Phi:0\to k\cdot t\to B\to A\to 0$. We identify locally trivial deformations and non-abelian \v{C}ech cohomologies and utilize certain exact sequences (cf. \cite[Definition 11.11]{gw}). The problem of existence of a locally trivial lift is decomposed into two parts, which we can interpret as follows: First, we determine whether lifts of gluing morphisms exist (Proposition \ref{classes interpret prop} (i)). Second, we determine whether we can take the lifts in such a way that they satisfy the cocycle conditions (Proposition \ref{classes interpret prop} (ii)). We use $\operatorname{Ob}_{X_0}^\Phi$ to describe these problems. 
Furthermore, we analyze the fiber of a locally trivial deformation $X\in \operatorname{Def}'_{X_0}(A)$ along $\alpha: \operatorname{Def}'_{X_0}(B)\to \operatorname{Def}'_{X_0}(A)$. Determining the complete structure of such fibers is difficult, even in characteristic $0$ or even if $X_0$ is smooth. However, at least we have the following:
\begin{theorem}[= Theorem \ref{lt main thm}]\label{main thm 3}
    We have an isomorphism of sets
    $$\alpha^{-1}(X)/H^1(\operatorname{Der}_k(\mathcal{O}_{X_0}))\simeq \operatorname{ker}\nu _X^\Phi/L(X).$$
The symbols $\nu _X^\Phi,\ L(X)$ will be defined in terms of $H^0(\operatorname{Ob}^\Phi_{X_0})$ in Definition \ref{L(X) def}, \ref{nu def}.
\end{theorem} 
On the left hand side, we take the quotient of the group action by $H^1(\operatorname{Der}_k(\mathcal{O}_{X_0}))$; see \cite[Remark 2.15]{Sch}. Note that the above isomorphic sets are either empty or a single point when $\operatorname{Def}'_{X_0}$ satisfies Schlessinger's $(H_1)$ (see the proof of \cite[Theorem 2.11]{Sch}). Thus, these sets measure the degree of non-pro-representability of $\operatorname{Def}'_{X_0}$ in a sense. We note that, in characteristic $0$, we have $\operatorname{Ob}^\Phi_{X_0}=0$ and in particular $\operatorname{ker}\nu _X^\Phi/L(X)$ cannot consist of more than one point. This also explains why $\operatorname{Def}'_{X_0}$ satisfies $(H_1)$ in characteristic $0$.

In Section 6, we construct an algebraic curve $C_0$ whose locally trivial deformation functor does not satisfy Schlessinger's $(H_1)$. The construction itself is valid over any field of positive characteristic. To confirm that $\operatorname{Def}'_{C_0}$ does not satisfy $(H_1)$, the above Theorem \ref{main thm 3} is used in an essential way in the case when the small extension is $\Pi:0\to k\cdot \lambda^p\to k[\lambda]\overset{}{\to} k[\varepsilon]\to 0$ where $\lambda^{p+1}=0,\ \varepsilon^p=0$ and $X=C$ is the trivial deformation of $C_0$. Instead of directly calculating $\alpha^{-1}(C)/H^1(\operatorname{Der}_k(\mathcal{O}_{X_0}))$ we study the above $L(C)$ based on calculating the global deformation automorphisms. In fact, we show that the inclusion $L(C)\subset H^0(\operatorname{Ob}_{C_0}^\Pi)$ is proper (Lemma \ref{inclusion proper lem}). Also, we have $\operatorname{ker}\nu _C^\Pi=H^0(\operatorname{Ob}_{C_0}^\Pi)$ in this case. Thus, the set $\operatorname{ker}\nu _C^\Pi/L(C)$ consists of more than two points, and we can prove that $\operatorname{Def}'_{C_0}$ does not satisfy $(H_1)$.

\textbf{Acknowledgements.} The author would like to thank Professor Keiji Oguiso and the members of Oguiso's laboratory for their interest in this work and their valuable comments and suggestions. The author would like to thank Professors Keiji Oguiso, Stefan Schröer, Xun Yu, Yujiro Kawamata and Gebhard Martin and the members of Oguiso’s laboratory for their valuable comments on the example in Section 6.

\section{Hasse-Schmidt derivations}

We first recall basic facts of Hasse-Schmidt derivations (cf. \cite{Has}, \cite[Chapter 1]{Mac}, \cite[Section 27]{Mat}) and results from \cite{Mac} and \cite{Her}.

\begin{definition}[= {\cite[Definition 1.2.1–1.2.6]{Mac}}] \label{Hasse Schmidt}
Let $R$ be a $k$-algebra. Let $m\ge 1$ be an integer. Then, a \emph{Hasse-Schmidt derivation} of length $m$ is a sequence $D=(D_i)_{i=0}^m$ of $k$-linear endomorphisms $R\to R$ such that $D_0=\operatorname{id}$ and that $$D_i(xy)= \sum_{j+k=i}D_j(x)D_k(y)$$ for every $x,y\in R, i\ge 1.$ Let $E=(E_i)_{i=0}^m$ be another Hasse-Schmidt derivation. We define the \emph{composition} of $D$ and $E$ by
$$D\circ E=(F_i)_{i=0}^m \mathrm{\ \ where \ } \ F_i(x)=\sum_{j+k=i}D_j(E_k(x)).$$
The symbol $\operatorname{HS}_k^m(R)(=\operatorname{HS}_k(R,m))$ denotes the set of all the Hasse-Schmidt derivations of length $m.$ It is a group whose multiplication is given by $\circ$. Let $x\in R$. We define 
$$ x \bullet: \operatorname{HS}_k^m(R)\to\operatorname{HS}_k^m(R)$$
by
$$x\bullet (D_i)_{i=0}^m=(x^iD_i)_{i=0}^m.$$
For $n\ge1,$ we define
$$[n]:\operatorname{HS}_k^m(R)\to\operatorname{HS}_k^{mn}(R)$$
by
$$D[n]=(E_i)_{i=0}^{mn}\mathrm{\ \ where\ }\ E_{jn}=D_j \mathrm{\ for\ }0\le j\le m \mathrm{\ and \ } E_i=0\mathrm{\ otherwise.}$$
A Hasse-Schmidt derivation $D=(D_i)_{i=0}^m$ defines a natural $k$-algebra homomorphism $$R\to R[\varepsilon]/\varepsilon^{m+1},\ x\mapsto \sum_{i=0}^m D_i(x)\varepsilon^i.$$ 
Let $1\le l \le m$. We define the \emph{truncation map} $$\tau:\operatorname{HS}_k^m(R)\to \operatorname{HS}_k^l(R)$$ by $$(D_i)_{i=0}^m\mapsto(D_i)_{i=0}^l.$$ It is a group homomorphism and commutes with $x\bullet$ for $x\in R$.
Note that the set $\operatorname{HS}_k^1(R)$ can be identified with the $R$-module $\operatorname{Der}_k(R,R)$ under $$D=(\operatorname{id},D_1)\mapsto D_1.$$
\end{definition}

\begin{definition}[= {\cite[Definition 2.1.1]{Mac}}] \label{m- inf- integrable der}Let $R$ be a $k$-algebra and let $1\le m < \infty$. We define the module of\emph{ $m$-integrable derivations }$\operatorname{Der}_k^m(R)$ as the image of
$$\operatorname{HS}_k^m(R)\overset{\tau}{\to}\operatorname{HS}_k^1(R)\simeq \operatorname{Der}_k(R,R).$$ It is a sub-$R$-module of $\operatorname{Der}_k(R,R)$.
\end{definition}

\begin{lemma}\label{n m lift} Let $1\le n \le m < \infty $ be integers and $T$ be a formally smooth $k$-algebra. Then, any Hasse-Schmidt derivation $D\in \operatorname{HS}_k^n(T)$ lifts to $E\in \operatorname{HS}_k^m(T)$.
\end{lemma}
\begin{proof}
This corresponds to the existence of the dotted arrow in the following diagram
\[\begin{tikzcd}
	&&& T \\
	0 & {(\varepsilon^{n+1})} & {T[\varepsilon]/\varepsilon^{m+1}} & {T[\varepsilon]/\varepsilon^{n+1}} & 0
	\arrow[dashed, from=1-4, to=2-3]
	\arrow[from=1-4, to=2-4]
	\arrow[from=2-1, to=2-2]
	\arrow[from=2-2, to=2-3]
	\arrow[from=2-3, to=2-4]
	\arrow[from=2-4, to=2-5]
\end{tikzcd}\]
where the vertical arrow is induced from $D$.
\end{proof}
\begin{lemma}\label{smoothness} Let $m\ge1 $ be an integer, $T$ be a formally smooth $k$-algebra, $I\subset  T$ be an ideal and $R:=T/I$. Then, any $D\in \operatorname{HS}_k^m(R)$ lifts to $E\in \operatorname{HS}_k^m(T)$.
\end{lemma}
\begin{proof}
Let $S=T[\varepsilon]/\varepsilon^{m+1}$ and let $J\subset S$ be
the ideal generated by $x\varepsilon, x\in I$. The $k$-linear map
$$f:T\to S/J, \ x\mapsto x+ \sum_{i=1}^m D_i(\overline{x})\varepsilon^i $$
is a $k$-algebra homomorphism, where $\overline{x}$ is the image of $x$ in $R$. Because $T$ is formally smooth and because $J$ is nilpotent, $f$ lifts to 
$\widetilde{f}:T\to S=T[\varepsilon]/\varepsilon^{m+1}$
and $\widetilde{f}$ induces $E\in \operatorname{HS}_k^m(T)$ that is a lift of $D$.
\end{proof}
Let $R$ be a $k$-algebra, $m\ge 1$ an integer and $S\subset R$ a multiplicatively closed subset. Then, we have a canonical map
$$\operatorname{HS}_k^m(R)\to\operatorname{HS}_k^m(S^{-1}R)$$
and it induces$$\alpha: S^{-1}\operatorname{Der}_k^m(R)\to \operatorname{Der}_k^m(S^{-1}R).$$

Now, let us prove a theorem by Macarro that states that $m$-integrable derivations are compatible with localizations.

\begin{theorem}[= {\cite[Corollary 2.3.5]{Mac}}]
\label{localization}
    Let $R$ be a $k$-algebra essentially of finite type over $k$. Let $m\ge 1$ an integer. Let $S\subset R$ a multiplicatively closed subset. Then, we have a canonical isomorphism $$\alpha: S^{-1}\operatorname{Der}_k^m(R)\overset{\sim}{\to} \operatorname{Der}_k^m(S^{-1}R).$$
\end{theorem}
\begin{proof}
As $S^{-1}\operatorname{Der}_k(R)\to \operatorname{Der}_k(S^{-1}R)$ is an isomorphism, $\alpha$ is injective. For surjectivity, let us take $D=(D_i)_{i=0}^m\in\operatorname{HS}_k^m(S^{-1}R)$. Let $T$ be a localization of $k[X_1,\dots,X_n]$ and $I\subset T$ an ideal such that $T/I=R.$ Let $S'\subset T$ be the preimage of $S$. There exists $D'=(D_i')_{i=0}^m\in\operatorname{HS}_k^m(S'^{-1}T)$ such that $$D'\operatorname{mod}S'^{-1}I=D.$$ As $m$ and $n$ are finite, there exists $s_1\in S'$ such that $s_1\bullet D'$ is induced from $\widetilde{D}\in \operatorname{HS}_k^m(T)$. Again, as $m$ and $n$ are finite and $I$ is finitely generated, there exists $s_2\in S'$ such that $s_2\bullet \widetilde{D}=\widetilde{E}$ satisfies $\widetilde{E}_i(I)\subset I$ for every $0 \le i \le m$. Then, $E= \widetilde{E} \operatorname{mod} I$ is an element of $\operatorname{HS}_k^m(R)$ and $(s_1s_2)^{-1}(E_1)=D_1$. This shows the surjectivity of $\alpha$.
\end{proof}

The following definition is inspired by \cite[Lemma 2.10]{Her}.
\begin{definition}\label{xi operator}
Let $R$ be a $k$-algebra. Let $m\ge 1$ and $n\ge2$ integers such that $p$ does not divide $n$. Let us take a primitive $n$-th root of unity $\zeta_n\in k$. We define a map
$$\chi_n: \operatorname{HS}_k^m(R)\to \operatorname{HS}_k^m(R) $$
by
$$\chi_n(D)=(\frac{\zeta_n^{n-1}}{\sqrt[n]{n}}\bullet D)\circ(\frac{\zeta_n^{n-2}}{\sqrt[n]{n}}\bullet D)\circ\dots\circ (\frac{\zeta_n^{0}}{\sqrt[n]{n}}\bullet D).$$
\end{definition}

\

Although the map $\chi_n$ depends on the choice of $\zeta_n$ and $\sqrt[n]{n}$, any such choice will work in the following arguments. Note that $\chi_n$ commutes with truncation maps.

\begin{lemma}\label{xi operator lem}
In the setting of Definition \ref{xi operator}, let us assume $n\nmid m$. Let $D\in \operatorname{HS}_k^m(R)$ be such that, for every $1\le l < m$ such that $n\nmid l$, we have $D_l=0$. Then, for $E=\chi_n(D),$ we have $E_m=0$.
\end{lemma}

\begin{proof}
    We have
\begin{align*}
    E_m&=(\frac{\zeta_n^{n-1}}{\sqrt[n]{n}})^m\cdot D_m+(\frac{\zeta_n^{n-2}}{\sqrt[n]{n}})^m\cdot D_m+\dots+(\frac{\zeta_n^{0}}{\sqrt[n]{n}})^m\cdot D_m \\
    &=(\zeta_n^{m\cdot(n-1)}+ \zeta_n^{m\cdot(n-2)}+\dots+\zeta_n^{m\cdot 0})\cdot \frac{1}{n^{\frac{m}{n}}}D_m \\
    &=0.\end{align*}\end{proof}
    
\begin{lemma}\label{xi operator lem 2}
In the setting of Definition \ref{xi operator}, let $D\in \operatorname{HS}_k^m(R)$ and $E=\chi_n(D)$. Then, for every $i\ge1$ such that $p^in\le m$, we have
$$E_{p^in}=D_{p^in}+ \mathrm{(noncommutative\ polynomial\ in\ } D_1,\dots, D_{p^in-1}).$$
\end{lemma}
\begin{proof} The coefficient of $D_{p^in}$ in $E_{p^in}$ is
    \begin{align*}
        &(\frac{\zeta_n^{n-1}}{\sqrt[n]{n}})^{p^in}+(\frac{\zeta_n^{n-2}}{\sqrt[n]{n}})^{p^in}+\dots+(\frac{\zeta_n^{0}}{\sqrt[n]{n}})^{p^in}\\
        =&(\zeta_n^{{p^in}\cdot(n-1)}+ \zeta_n^{{p^in}\cdot(n-2)}+\dots+\zeta_n^{{p^in}\cdot 0})\cdot \frac{1}{n^{\frac{{p^in}}{n}}}\\
        =&((\zeta_n^{{n}\cdot(n-1)}+ \zeta_n^{{n}\cdot(n-2)}+\dots+\zeta_n^{{n}\cdot 0})\cdot \frac{1}{n})^{p^i}\\
        =&1^{p^i}\\
        =&1. 
    \end{align*}
    \end{proof}

\begin{lemma}\label{xi operator lem 3}In the setting of Definition \ref{xi operator}, let us assume that $m=p^in$ with $ i\ge0$ and $ n\ge2$. Then for any $D\in \operatorname{HS}_k^m(R)$, there exists $D^{[n]}\in \operatorname{HS}_k^m(R)$ such that $D^{[n]}_l=0$ for every $1\le l < m$ with $n\nmid l$, and $$D^{[n]}_m=D_m+ \mathrm{(noncommutative\ polynomial\ in\ } D_1,\dots, D_{m-1}).$$
\end{lemma}
\begin{proof}Apply $\chi_n$ to $D$ iteratively sufficiently many times and denote the resulting element by $D^{[n]}\in \operatorname{HS}_k^m(R)$. By Lemma \ref{xi operator lem} and \ref{xi operator lem 2}, $D^{[n]}$ satisfies the desired properties. \end{proof}
Now, we can give an easier proof of a theorem by Hernández under the assumption that $k$ is algebraically closed.

\begin{theorem}[= {\cite[Theorem 4.1]{Her}}]\label{p^n leap} Let $R$ be a $k$-algebra essentially of finite type. For every $m\ge 2,$ the inclusion $\operatorname{Der}_k^m(R)\subset \operatorname{Der}_k^{m-1}(R) $ is proper only if $m$ is of the form $m=p^i$ for $i\ge 1$.
\end{theorem}
\begin{proof}
Let us assume $m=p^in,i\ge0,n\ge2$ where $p\nmid n$. It suffices to show that $\operatorname{Der}_k^m(R)=\operatorname{Der}_k^{m-1}(R)$. Let us take $D\in \operatorname{HS}_k^{m-1}(R)$. We would like to show that $D_1\in \operatorname{Der}_k^m(R)$. Let $T$ be a localization of a polynomial algebra over $k$ and $I\subset T$ an ideal such that $T/I=R.$ Then, $D$ lifts to $E'\in \operatorname{HS}_k^{m-1}(T)$ and $E'$ lifts to $E\in \operatorname{HS}_k^m(T)$. Let us apply Lemma \ref{xi operator lem 3} to $E$ and let $E^{[n]}$ denote the resulting element. We have $E^{[n]}_l=0$ for every $1\le l < m$ satisfying $n\nmid l$, in particular $E^{[n]}_1=0$, and we have $$E^{[n]}_m=E_m+ \mathrm{( noncommutative\ polynomial\ in\ } E_1,\dots, E_{m-1}).$$
Let us define $G:=(E^{[n]})^{-1}\circ E$. We have $G_1=E_1$ and for every $1\le l\le m$, $G_l$ is a noncommutative polynomial in $E_1,\dots, E_{m-1}$ (the term $E_m$ cancels out in $G_m$). In particular, $G_l(I)\subset I$ for every $l$ and $G$ induces $\overline{G}\in \operatorname{HS}_k^m(R)$, so that $\overline{G}_1=D_1\in \operatorname{Der}_k^m(R)$.
\end{proof}
\begin{remark}
    The result of Lemma \ref{xi operator lem 3} is superfluous in the above proof as it uses only that $E^{[n]}_1=0$. We are going to use Lemma \ref{xi operator lem 3} in an essential way in Proposition \ref{np=p}.
\end{remark}

Let us recall the {\em first cotangent module} $T^1_{R/k}$ of a $k$-algebra $R$ essentially of finite type (cf. \cite[Definition 1.1.6]{Ser}). Let $T$ be a localization of a polynomial algebra over $k$ and $I\subset T$ an ideal such that $T/I=R.$ By \cite[Corollary 1.1.8]{Ser}, we have the following canonical exact sequence of $R$-modules

$$\operatorname{Hom}_R(\Omega_{T/k}\otimes_TR,R)\to \operatorname{Hom}_R(I/I^2,R)\to T^1_{R/k}\to 0.$$

\begin{definition}\label{obstruction sp M,Ob}
    Let $R$ be a $k$-algebra essentially of finite type and $m\ge 2 $ an integer. We define $$\operatorname{ob}_m:\operatorname{HS}_k^{m-1}(R)\to T^1_{R/k}$$
    as follows: Let $T$ be a localization of a polynomial algebra over $k$ and $I\subset T$ an ideal such that $T/I=R.$ For $D\in \operatorname{HS}_k^{m-1}(R)$, let us take its lift $E\in \operatorname{HS}_k^m(T)$ as in the proof of Theorem \ref{p^n leap}. Then, $E$ induces a $k$-algebra homomorphism $$\varphi_E :T\to R[\varepsilon]/\varepsilon^{m+1}.$$ We have $\varphi_E(I)\subset\varepsilon^m\cdot R[\varepsilon]/\varepsilon^{m+1}$ and $ \varphi_E(I^2)=0$ so that $\varphi_E$ induces $$f_E:I/I^2\to \varepsilon^m\cdot R[\varepsilon]/\varepsilon^{m+1}\simeq R.$$
 We define $\operatorname{ob}_m(D)$ as the image of $f_E$ under $\operatorname{Hom}_R(I/I^2,R)\to T^1_{R/k}.$ 
\end{definition}

\begin{proposition}\label{obstruction independence} In the situation of Definition \ref{obstruction sp M,Ob}, the map $\operatorname{ob}_m$ is a group homomorphism independent of the choice of $T, I, E$. A Hasse-Schmidt derivation $D\in \operatorname{HS}_k^{m-1}(R)$ extends to an element in $\operatorname{HS}_k^{m}(R)$ if and only if $\operatorname{ob}_m(D)=0$. For every $D\in \operatorname{HS}_k^{m-1}(R)$ and $x\in R$, we have $$\operatorname{ob}_m(x\bullet D)=x^m\cdot \operatorname{ob}_m(D).$$
\end{proposition}
\begin{proof}First, let us show that $\operatorname{ob}_m$ is a group homomorphism. Let $D^1,D^2\in \operatorname{HS}_k^{m-1}(R)$ be Hasse-Schmidt derivations and let $E^1,E^2\in \operatorname{HS}_k^m(T)$ be their lifts. Let $J\subset T[\varepsilon]/\varepsilon^{m+1}$ be the ideal generated by $\varepsilon\cdot I$. Note that $$(T[\varepsilon]/\varepsilon^{m+1})/J=T\oplus \varepsilon R\oplus \dots \oplus \varepsilon^m R.$$ Let $$\psi_1,\psi_2:T[\varepsilon]/\varepsilon^{m+1}\rightrightarrows T[\varepsilon]/\varepsilon^{m+1}$$ be the induced morphisms from $E^1,E^2$. Note that $\psi_1(I),\psi_2(I)\subset (I,\varepsilon^m)$ because they are lifts of $R[\varepsilon]/\varepsilon^m\rightrightarrows R[\varepsilon]/\varepsilon^m$. Thus, $\psi_1(J),\psi_2(J)\subset J$ and there are induced morphisms $$\overline{\psi}_1,\overline{\psi}_2: (T[\varepsilon]/\varepsilon^{m+1})/J\rightrightarrows (T[\varepsilon]/\varepsilon^{m+1})/J$$ from $E^1,E^2$. For $a\in I,$ we have equalities in $(T[\varepsilon]/\varepsilon^{m+1})/J$:
    \begin{align*}
    \varepsilon^m\cdot f_{E^1\circ E^2}(a+I^2)&=(\overline{\psi}_1\circ \overline{\psi}_2)(a)-a\\
    &=\overline{\psi}_1(a+\varepsilon^m\cdot f_{E^2}(a+I^2))-a\\
    &=a+\varepsilon^m\cdot f_{E^1}(a+I^2)+\varepsilon^m\cdot f_{E^2}(a+I^2)-a\\
    &=\varepsilon^m\cdot (f_{E^1}+ f_{E^2})(a+I^2). \end{align*}
This implies $f_{E^1\circ E^2}=f_{E^1}+ f_{E^2}$ and $\operatorname{ob}_m(D^1\circ D^2)=\operatorname{ob}_m(D^1)+\operatorname{ob}_m(D^2)$, so $\operatorname{ob}_m$ is a group homomorphism.
Now, let us show that the element $\operatorname{ob}_m(D)$ does not depend on the choice of $T,I$ nor $E$. First, fix $T$ and $I$ and let $F\in \operatorname{HS}_k^m(T)$ be another lift of $D$. Then, $E$ and $F$ induce $k$-algebra homomorphisms 
$$\varphi_E,\varphi_F :T\rightrightarrows R[\varepsilon]/\varepsilon^{m+1}$$ 
such that $\varphi_E=\varphi_F$ modulo $\varepsilon^m.$  Hence, $\varphi_E-\varphi_F$ is a $k$-derivation from $T$ to $\varepsilon^m\cdot R[\varepsilon]/\varepsilon^{m+1}$. This means that $f_E=f_F$ in $T^1_{R/k}.$ A similar argument also shows that $D$ extends to an element of $\operatorname{HS}_k^{m}(R)$ if and only if $\operatorname{ob}_m(D)=0$. Now, let $T',I'$ be another choice of $T,I$. Then, there exists $\psi: T'\to T$ that induces $T'/I'\simeq T/I$. Note that $\varphi_E\circ \psi$ is induced from some $E'\in \operatorname{HS}_k^m(T')$ (by an argument similar to that of Lemma \ref{smoothness}.) Note that $\psi$ induces $\widetilde{\psi}
:I'/I'^2\to I/I^2$. We have to show that $f_E$ and $f_E\circ \widetilde{\psi}$ induce the same element in $T^1_{R/k}$, and this follows from the commutative diagram 
\[\begin{tikzcd}
	{\operatorname{Hom}_R(I/I^2,R)} & {T^1_{R/k}} \\
	{\operatorname{Hom}_R(I'/I'^2,R)} & {T^1_{R/k}}
	\arrow[from=1-1, to=1-2]
	\arrow["{\widetilde{\psi}^{\sharp}}"', from=1-1, to=2-1]
	\arrow[equals, from=1-2, to=2-2]
	\arrow[from=2-1, to=2-2].
\end{tikzcd}\]
Finally, the equality $\operatorname{ob}_m(x\bullet D)=x^m\cdot \operatorname{ob}_m(D)$ follows from the construction.
\end{proof}

\begin{definition}\label{Ob}
  Let $R$ be a $k$-algebra essentially of finite type and $m\ge 2 $ an integer.  We define the $m$-th \emph{obstruction module} $\operatorname{Ob}^m_R\subset T^1_{R/k}$ as the image of $\operatorname{ob}_m$. At this point, this is an abelian group.
\end{definition}

\begin{example}
    If $R$ is formally smooth, then we have $\operatorname{Ob}_R^m=0$ for every $m\ge 2$ by Lemma \ref{n m lift}.
\end{example}

\begin{proposition}\label{Ob module}
In the situation of Definition \ref{obstruction sp M,Ob}, let us assume $m=p^i, i\ge1$. Then, $\operatorname{Ob}^{p^i}_R$ is a finitely generated $R^{p^i}$-module.  If $S\subset R$ is a multiplicatively closed subset, then $(S^{p^i})^{-1}\operatorname{Ob}_R^{p^i}=\operatorname{Ob}^{p^i}_{S^{-1}R}$. 
\end{proposition}

\begin{proof}
    As $\operatorname{ob}_{p^i}(x\bullet D)=x^{p^i}\cdot \operatorname{ob}_m(D)$, $\operatorname{Ob}_R^{p^i}$ is a $R^{p^i}$-module. It is finitely generated because $R$ is finite over $R^{p^i}$ (as $k$ is perfect) and $\operatorname{Ob}^{p^i}_R$ is a sub-$R^{p^i}$-module of $T^1_{R/k}$. Note that we have the following commutative diagram
\[\begin{tikzcd}
	{\operatorname{HS}_k^{p^i-1}(R)} & {T^1_{R/k}} \\
	{\operatorname{HS}_k^{p^i-1}(S^{-1}R)} & {T^1_{S^{-1}R/k}}
	\arrow["{\operatorname{ob}_{p^i}}"', from=1-1, to=1-2]
	\arrow["u", from=1-1, to=2-1]
	\arrow[from=1-2, to=2-2]
	\arrow["{\operatorname{ob}_{p^i}}", from=2-1, to=2-2]
\end{tikzcd}\]
and this induces $(S^{p^i})^{-1}\operatorname{Ob}_R^{p^i}\to \operatorname{Ob}_{S^{-1}R}^{p^i}$, which is injective because
$$(S^{p^i})^{-1}T^1_{R/k}\simeq S^{-1}T^1_{R/k}\to T^1_{S^{-1}R/k}$$ is an isomorphism. For every $D\in \operatorname{HS}_k^{p^i-1}(S^{-1}R)$, there exists $s\in S$ such that there exists $E\in \operatorname{HS}_k^{p^i-1}(R)$ with $u(E)=s\bullet D$, as in the proof of Theorem \ref{localization}. We have $$\operatorname{ob}_{p^i}(D)=(s^{p^i})^{-1}\cdot \operatorname{ob}_{p^i}(E)$$ in $\operatorname{Ob}_{S^{-1}R}^{p^i}$ and this shows surjectivity.
\end{proof}
\begin{lemma}
    \label{Ob module 2}
    In the situation of Definition \ref{obstruction sp M,Ob}, suppose that we have another integer $n\ge2$. Then, we have $\operatorname{ob}_{m}(D)=\operatorname{ob}_{mn}(D[n])$ for $D\in \operatorname{HS}_k^{m-1}(R)$ and in particular $\operatorname{Ob}^m_R\subset \operatorname{Ob}_R^{mn}.$
    Here, we think of $D[n]$ as an element of $\operatorname{HS}_k^{mn-1}(R)$ by defining $(D[n])_l=0$ for $l=mn-n+1,\ mn-n+2,\dots,\ mn-1$.
\end{lemma}
\begin{proof}
    If $E\in \operatorname{HS}_k^{m}(T)$ is a lift of $D$, then $E[n]\in \operatorname{HS}_k^{mn}(T)$ is a lift of $D[n]$. As $E_m=E[n]_{mn}$, we have $\operatorname{ob}_{m}(D)=\operatorname{ob}_{mn}(D[n])$.
\end{proof}

\begin{proposition}\label{np=p}
In the situation of Definition \ref{obstruction sp M,Ob}, let us assume $m=p^in,i\ge0,n\ge2$ where $p\nmid n$. Then, we have $\operatorname{Ob}^m_R=\operatorname{Ob}^{p^i}_R$.
\end{proposition}
\begin{proof} Let $D\in \operatorname{HS}_k^{m-1}(R)$ and let us take $T$ and $E\in \operatorname{HS}_k^m(T)$ as in Definition \ref{obstruction sp M,Ob}.
Let us apply Lemma \ref{xi operator lem 3} to $E$ and let $E^{[n]}$ denote the resulting element. We have $E^{[n]}_l=0$ for every $1\le l < m$ satisfying $n\nmid l$,  and we have $$E^{[n]}_m=E_m+ \mathrm{(noncommutative\ polynomial\ in\ } E_1,\dots, E_{m-1}).$$This means that $E^{[n]}$ induces $D'\in \operatorname{HS}_k^{p^i-1}(R)$ such that $\operatorname{ob}_m(D)=\operatorname{ob}_{p^i}(D')$, and this proves that $\operatorname{Ob}_R^m\subset \operatorname{Ob}_R^{p^i}.$ We have shown the opposite inclusion in Lemma \ref{Ob module 2}.
\end{proof}

By Propositions \ref{Ob module} and \ref{np=p}, we obtain:

\begin{theorem}\label{Ob sheaf}
    Let $X$ be a scheme of finite type over $k$. For each $m\ge 1$, there exists a subsheaf
    $$\operatorname{Ob}_X^{m}\subset T_X^1$$
    such that for every affine open subscheme $\operatorname{Spec}R\subset X,$ we have
    $$\operatorname{Ob}_X^{m}(\operatorname{Spec}R)=\operatorname{Ob}_R^m\subset T_R^1.$$ It is a sub-$\mathcal{O}_X^{p^i}$-module of $T_X^1$, where $i$ is the largest integer such that $p^i$ divides $m$. It is also a coherent $\mathcal{O}_X$-module via $\mathcal{O}_X\to \mathcal{O}_X^{p^i}$.
    
\end{theorem}

\begin{definition}\label{M} Let $R$ be a $k$-algebra essentially of finite type. We define the $i$-th \emph{filtered obstruction module} $\operatorname{Fob}_R^i$ by
$$\operatorname{Fob}_R^1=\operatorname{Ob}_R^{p}$$
for $i=1,$
and
$$\operatorname{Fob}_R^i=\operatorname{Ob}_R^{p^i}/\operatorname{Ob}_R^{p^{i-1}}$$
 for $i=2,3,4,\dots.$
We regard each $\operatorname{Fob}_R^i$ as a finite $R$-module via $R\to R^{p^i}$. This is compatible with localization. Similarly, if $X$ is an algebraic scheme over $k$, then we define
$$\operatorname{Fob}_X^1=\operatorname{Ob}_X^{p}$$
for $i=1,$
and
$$\operatorname{Fob}_X^i=\operatorname{Ob}_X^{p^i}/\operatorname{Ob}_X^{p^{i-1}}$$
 for $i=2,3,4,\dots,$ which are coherent $\mathcal{O}_X$-modules.
\end{definition}

\begin{lemma}\label{x_n y_n}
    Let $n\ge2$ an integer such that $p\nmid n.$ Then, there exist $x_{n},y_n\in k$ such that $x_n+y_n=1$ and $x_n^n+y^n_n=0$.
\end{lemma}

\begin{definition}\label{omega operator}Let $R$ be a $k$-algebra.
    Let $n\ge2$ an integer. Using $x_n$, $y_n$ in Lemma \ref{x_n y_n}, we define a map $$\omega_n: \operatorname{HS}_k^m(R)\to \operatorname{HS}_k^m(R) $$
by
$$\omega_n(D)=(x_n\bullet D)\circ(y_n\bullet D)$$
if $p\nmid n$, and by
$$\omega_n=\operatorname{id}$$
if $p\mid n$.

\end{definition}
The following two lemmas can be proven similarly as Lemma \ref{xi operator lem} and \ref{xi operator lem 2}.
\begin{lemma}\label{omega operator lem}
In the setting of Definition \ref{omega operator}, let us suppose $n\le m$ and $p\nmid n$. Let $D\in \operatorname{HS}_k^m(R)$ be such that, for every $1\le l < n$ with $p\nmid l$, we have $D_l=0$. Then, for $E=\omega_n(D),$ we have $E_n=0$.
\end{lemma}
\begin{lemma}\label{omega operator lem 2}
In the setting of Definition \ref{omega operator}, let $D\in \operatorname{HS}_k^m(R)$ and $E=\omega_n(D)$. Then, for every $i\ge1$ such that $p^i\le m$, we have
$$E_{p^i}=D_{p^i}+ \mathrm{(noncommutative\ polynomial\ in\ } D_1,\dots, D_{p^i-1}).$$
\end{lemma}
\begin{proposition}\label{kernel circ}
Let $R$ be a $k$-algebra essentially of finite type. Let $m=p^i$ with $i\ge1$. Let $\operatorname{HS}_k^{p^i-1}(R)^\circ$ denote the subgroup of $\operatorname{HS}_k^{p^i-1}(R)$ consisting of elements $(D_i)_{i=0}^{p^i-1}$ such that $ D_1=0$. Then, $\operatorname{HS}_k^{p^i-1}(R)^\circ$ maps to zero under the composition $$\operatorname{HS}_k^{p^i-1}(R)^\circ\hookrightarrow \operatorname{HS}_k^{p^i-1}(R) \overset{\operatorname{ob}_m}{\longrightarrow}\operatorname{Ob}_R^{p^i}\twoheadrightarrow \operatorname{Fob}^i_R.$$

\begin{proof}
    Let $T$ be a localization of a polynomial algebra over $k$, and $I\subset T$ an ideal such that $T/I=R.$ For $D\in \operatorname{HS}_k^{m-1}(R)^\circ$, let us take its lift $E\in \operatorname{HS}_k^m(T)^\circ$ as in the proof of Theorem \ref{p^n leap}. Let 
    $$F:=\omega_{m-1}(\omega_{m-2}(\dots(\omega_2(E))\dots)).$$ Then, for $1\le l \le m,$ we have $F_l\neq 0$ only if $p\mid l$, and $$F_m=E_m+\mathrm{(noncommutative \ polynomial \ in\ }E_1,\dots,E_{m-1}).$$ This implies that there exists $\widetilde{D}\in \operatorname{HS}_k^{p^{i-1}-1}(R)$ such that $$\operatorname{ob}_{p^{i}}(D)=\operatorname{ob}_{p^{i}}(\widetilde{D}[p])=\operatorname{ob}_{p^{i-1}}(\widetilde{D}).$$ (See Lemma \ref{Ob module 2} for the definition of $\widetilde{D}[p]$.) This proves the assertion.
\end{proof}
\end{proposition}

\begin{proposition}\label{exact seq}
Let $R$ be a $k$-algebra essentially of finite type. For every $i=1,2,3,\dots,$ we have a well-defined homomorphism of $R$-modules 
$$\operatorname{Der}_k^{p^{i-1}}(R)\to \operatorname{Fob}^i_R,$$
and
the sequence
$$0\to\operatorname{Der}_k^{p^i}(R)\overset{\alpha}\longrightarrow\operatorname{Der}_k^{p^{i-1}}(R)\overset{\beta}\longrightarrow \operatorname{Fob}^i_R\to0$$
is exact. 

\end{proposition}
\begin{proof} As we have an isomorphism $\operatorname{HS}_k^{p^i-1}(R)/\operatorname{HS}_k^{p^i-1}(R)^\circ\simeq \operatorname{Der}_k^{p^{i}-1}(R)$ as groups, the homomorphism is well-defined.
    We are left to prove that $\operatorname{ker}\beta \subset \operatorname{im}\alpha$. For this, let us take $D\in \operatorname{HS}_k^{p^i-1}(R) $ such that $\operatorname{ob}_{p^i}(D)\in \operatorname{Ob}_R^{p^{i-1}} $ and it suffices to show that $D_1\in \operatorname{Der}_k^{p^i}(R).$ By assumption, there exists $E\in \operatorname{HS}_k^{p^{i-1}-1}(R)$ such that $\operatorname{ob}_{p^{i-1}}(E)=-\operatorname{ob}_{p^i}(D)$. Then, for $F:=E[p]\circ D\in \operatorname{HS}_k^{p^i-1}(R) $, we have $\operatorname{ob}_{p^i}(F)=0$ and $F_1=D_1$, so that $D_1\in \operatorname{Der}_k^{p^i}(R)$. \end{proof}




By Proposition \ref{kernel circ} and a remark made in Definition \ref{M}, we obtain the following important theorems.
\begin{theorem}\label{duality}
Let $X$ be a scheme of finite type over $k$. Then, there exist two filtrations of sheaves on $X$ such that
$$\operatorname{Der}_k(\mathcal{O}_X)\supset\operatorname{Der}_k^p(\mathcal{O}_X)\supset\operatorname{Der}_k^{p^2}(\mathcal{O}_X)\supset\dots$$
and that
$$0\subset \operatorname{Ob}^p_X\subset\operatorname{Ob}^{p^2}_X\subset\dots\subset T^1_{X/k}.$$ For every $i=1,2,\dots,$ we have
$$\operatorname{Der}_k^{p^{i-1}}(\mathcal{O}_X)/\operatorname{Der}_k^{p^i}(\mathcal{O}_X)\simeq \operatorname{Fob}_X^i= \operatorname{Ob}_X^{p^i}/\operatorname{Ob}_X^{p^{i-1}}. $$
\end{theorem}

\begin{theorem}\label{scheme version}
Let $X$ be a scheme of finite type over $k$. For each $i=1,2,\dots$, we have a long exact sequence of $k$-vector spaces
    \begin{align*}  0&\to H^0(\operatorname{Der}_k^{p^{i}}(\mathcal{O}_X))\longrightarrow H^0(\operatorname{Der}_k^{p^{i-1}}(\mathcal{O}_X))\longrightarrow H^0(\operatorname{Fob}_X^i)\to \\ &\to H^1(\operatorname{Der}_k^{p^{i}}(\mathcal{O}_X))\longrightarrow H^1(\operatorname{Der}_k^{p^{i-1}}(\mathcal{O}_X))\longrightarrow H^1(\operatorname{Fob}_X^i)\to \dots\end{align*}
\end{theorem}

\section{Obstructions along small extensions}

\begin{definition} The symbol $\mathcal{A}$ denotes the category whose objects are local artinian $k$-algebras  $A=(A,\mathfrak{m}_A)$ with $A/\mathfrak{m}_A\simeq k$ and whose morphisms are local $k$-algebra homomorphisms.
\end{definition}


\begin{definition}
    Let $A$ be an object of $\mathcal{A}$ and $R$ be a $k$-algebra. Then, we define $\operatorname{HS}_k^A(R)(=\operatorname{HS}_k(R,A))$ as the group of $A$-algebra automorphisms $A \otimes _k R\to A\otimes_k R$ inducing the identity $ A/\mathfrak{m}_A\otimes_kR\to  A/\mathfrak{m}_A\otimes_kR$. (Note that this notation is not standard.)
\end{definition}

\begin{remark}
    Let $x\in R$ and $R_x$ be the localization of $R$ by $x$. If $D\in \operatorname{HS}_k^A(R)$, then the difference of $1\otimes x$ and $D(1\otimes x)$ is nilpotent.  Thus, $D$ induces
    $$A\otimes R_x \simeq (A\otimes R)_{1\otimes x}\to  (A\otimes R)_{D(1\otimes x)}\simeq A\otimes R_x.$$ This implies that we have a natural group homomorphism $\operatorname{HS}_k^A(R)\to \operatorname{HS}_k^A(R_x)$. More generally, if $\operatorname{Spec}R'$ is an affine open subscheme of $\operatorname{Spec}R$, then there is an induced natural group homomorphism $\operatorname{HS}_k^A(R)\to \operatorname{HS}_k^A(R').$
\end{remark}

\begin{definition}\label{basis def}
    Let $A$ be an object of $\mathcal{A}$ and $\lambda_1,\dots, \lambda_d$ be a $k$-linear basis of $\mathfrak{m}_A.$ Let $R$ be a $k$-algebra and $D\in \operatorname{HS}_k^A(R)$. Then, $\{D_{\lambda_i}\}_{1\le i\le d}$ denotes the family of $k$-linear endomorphisms $R\to R$ such that
    $$D(1\otimes r)=1\otimes r + \sum_i \lambda_i\otimes D_{\lambda_i}(r),$$ which is uniquely determined since $A\otimes_k R$ is a free $R$-module.
\end{definition}

\begin{remark}\label{basis identification} For $A=k[\varepsilon]/\varepsilon^{n+1},$ we can canonically identify $\operatorname{HS}_k^{A}(R)$ with $\operatorname{HS}_k^{n}(R)$ by taking $\lambda_1=\varepsilon,\lambda_2=\varepsilon^2,\dots$ as a basis of $\mathfrak{m}_{A}.$
    \end{remark}

\begin{definition}Let $f:A'\to A$ be a morphism in $\mathcal{A}$. By  $$f_*: \operatorname{HS}_k^{A'}(R)\to \operatorname{HS}_k^A(R),$$
we denote the natural group homomorphism given by $$f_*(D):=\operatorname{id}_{A}\otimes_{A'} D: A\otimes_{k} R\to A\otimes_{k} R.$$ Here we note that $A \otimes_k R \simeq A \otimes_{A'} (A' \otimes_k R)$.
\end{definition}

\begin{remark}Elements $D\in \operatorname{HS}_k^A(R)$ correspond bijectively to $k$-algebra morphisms $D|_R: R\to A\otimes_k R$ such that the composition
$$R \overset{D|_R}{\to} A\otimes_k R \overset{}{\to} (A/\mathfrak{m}_A)\otimes_k R\simeq R$$
is the identity. Here, $D|_R$ is given as the composition
$$R\simeq k\otimes_k R \overset{\iota\otimes \operatorname{id}}{\to} A\otimes_k R \overset{D}{\to} A\otimes_k R,$$ where $\iota :k\to A$ is the $k$-structure morphism.
\end{remark}

\begin{remark}\label{subring induce}
If $f:A'\hookrightarrow A$ is an inclusion in $\mathcal{A}$ and $D\in \operatorname{HS}_k^A(R)$ is such that the image of $D|_R$ is contained in $A'\otimes_k R$, then
$$D'=\operatorname{id}_{A'}\otimes_k (D|_R): A'\otimes_k R\to A'\otimes_k R$$ satisfies $f_*(D')=D$ by construction.
\end{remark}

Over general artinian rings, there is no analogue of the action $x\bullet$ on $\operatorname{HS}_k^n(R)$. We could define $x\bullet D$ as a map of sets, but it is not always a ring homomorphism.

\begin{definition}
    A \emph{small extension} (in $\mathcal{A}$) is a $k$-linear exact sequence
    $$\ 0\to k\cdot t\longrightarrow A' \overset{
    }{\longrightarrow}A \to 0$$
    where $A' \to A$ is a morphism in $\mathcal{A}$ and $0\neq t\in A'$ satisfies $\mathfrak{m}_{A'}\cdot t=0.$ The above small extension is called \emph{trivial} if $A' \to A$ has a section in $\mathcal{A}.$
\end{definition}
The following two lemmas correspond to Lemma \ref{n m lift} and Lemma \ref{smoothness}. Their proofs are similar.
\begin{lemma} Let $A'\to A$ be a surjection in $\mathcal{A}$ and $T$ be a formally smooth $k$-algebra. Then, any $D\in \operatorname{HS}_k^
{A}(T)$ lifts to some $E\in \operatorname{HS}_k^{A'}(T)$.
\end{lemma}
\begin{lemma}\label{artin smoothness} Let $A $ be an object of $\mathcal{A}$, $T$ be a formally smooth $k$-algebra, $I\subset T$ be an ideal and $R:=T/I$. Then, any $D\in \operatorname{HS}_k^A(R)$ lifts to $E\in \operatorname{HS}_k^A(T)$.
\end{lemma}

Now, we would like to generalize the notion of obstructions $\operatorname{ob}_m$ to arbitrary small extensions.

\begin{definition}\label{artin obstruction sp M,Ob}
    Let $R$ be a $k$-algebra essentially of finite type and $$\Phi: \ 0\to k\cdot t\longrightarrow A' \overset{
    }{\longrightarrow}A \to 0$$ be a small extension. We define $$\operatorname{ob}_\Phi:\operatorname{HS}_k^{A}(R)\to T^1_{R/k}$$
    as follows: Choose $T$, a localization of a polynomial algebra over $k$, and an ideal $I\subset T$ such that $T/I=R.$ For $D\in \operatorname{HS}_k^{A}(R)$, take a lift $E\in \operatorname{HS}_k^{A'}(T)$. Then $E$ induces a $k$-algebra homomorphism $$\varphi_E :T\overset{E|_T}{\to}A'\otimes_k T\to A'\otimes_kR.$$ We have $\varphi_E(I)\subset k\cdot t\otimes_k R, \ \varphi_E(I^2)=0$, and $\varphi_E$ induces $$f_E:I/I^2\to k\cdot t\otimes_k R\simeq R.$$
 We define $\operatorname{ob}_\Phi(D)$ as the image of $f_E$ under the map $\operatorname{Hom}_R(I/I^2,R)\to T^1_{R/k}.$ 
\end{definition}


\begin{proposition} In the situation of Definition \ref{artin obstruction sp M,Ob}, the map $\operatorname{ob}_\Phi$ is a group homomorphism independent of the choice of $T, I, E$. A Hasse-Schmidt derivation $D\in \operatorname{HS}_k^{A}(R)$ extends to an element of $\operatorname{HS}_k^{A'}(R)$ if and only if $\operatorname{ob}_\Phi(D)=0$.
\end{proposition}

\begin{proof}
We show that $\operatorname{ob}_\Phi$ is a group homomorphism, as in the proof of Proposition \ref{obstruction independence}. The other parts of the proof are also similar. Let $D^1,D^2\in \operatorname{HS}_k^{A}(R)$ be Hasse-Schmidt derivations and let $E^1,E^2\in \operatorname{HS}_k^{A'}(T)$ be their lifts. Let $J\subset A'\otimes_kT$ be the ideal $\mathfrak{m}_{A'}\otimes_k I$. We have $(A'\otimes_kT)/J\simeq k\otimes_k T\oplus \mathfrak{m}_{A'}\otimes_k R.$ Note that $E^1(1\otimes I),E^2(1\otimes I)\subset (1\otimes I,t\otimes T)$ because they are lifts of $A\otimes_k R\to A\otimes_kR$. Thus, $E^1(J),E^2(J)\subset J$ and there are induced morphisms $$\overline{E}^1,\overline{E}^2: (A'\otimes T)/J\rightrightarrows (A'\otimes T)/J$$ from $E^1,E^2$. For $a\in I,$ we have equalities in $(A'\otimes T)/J$:
    \begin{align*}
    t\cdot f_{E^1\circ E^2}(a+I^2)&=(\overline{E}^1\circ \overline{E}^2)(a)-a\\
    &=\overline{E}^1(a+t\cdot f_{E^2}(a+I^2))-a\\
    &=a+t\cdot f_{E^1}(a+I^2)+t\cdot f_{E^2}(a+I^2)-a\\
    &=t\cdot (f_{E^1}+ f_{E^2})(a+I^2). \end{align*}
This implies $f_{E^1\circ E^2}=f_{E^1}+ f_{E^2}$ and $\operatorname{ob}_\Phi(D^1\circ D^2)=\operatorname{ob}_\Phi(D^1)+\operatorname{ob}_\Phi(D^2)$, so that $\operatorname{ob}_\Phi$ is a group homomorphism.\end{proof}

\begin{definition}
  Let $R$ be a $k$-algebra essentially of finite type and $$\Phi: \ 0\to k\cdot t\longrightarrow A' \overset{
    }{\longrightarrow}A \to 0$$ be a small extension.  We define the \emph{obstruction module} along $\Phi$ to be the image of $\operatorname{ob}_\Phi$, which we denote by $\operatorname{Ob}^\Phi_R\subset T^1_{R/k}$. (At this point, it is only a subgroup. We will see that there is a module structure in Theorem \ref{artin obst theorem}.)
\end{definition}

\begin{proposition}
    Let $R$ be a $k$-algebra essentially of finite type and assume that $$\Phi: 0\to k\cdot t\to A'\to A\to 0,$$
    $$\Psi: 0\to k\cdot s\to B'\to B\to 0$$
are small extensions and that we have the following commutative diagram
\[\begin{tikzcd}
	0 & {k\cdot t} & {A'} & A & 0 \\
	0 & {k\cdot s} & {B'} & B & 0
	\arrow[from=1-1, to=1-2]
	\arrow[from=1-2, to=1-3]
	\arrow["{a\cdot}"', from=1-2, to=2-2]
	\arrow[from=1-3, to=1-4]
	\arrow["{g'}"', from=1-3, to=2-3]
	\arrow[from=1-4, to=1-5]
	\arrow["g"', from=1-4, to=2-4]
	\arrow[from=2-1, to=2-2]
	\arrow[from=2-2, to=2-3]
	\arrow[from=2-3, to=2-4]
	\arrow[from=2-4, to=2-5]
\end{tikzcd}\]
where $a\in k$. Then for any $D\in \operatorname{HS}_k^{A}(R),$ we have
$$a\cdot \operatorname{ob}_{\Phi}(D)=\operatorname{ob}_{\Psi}(g_*(D)).$$
\end{proposition}
\begin{proof}
    In the notation of Definition \ref{artin obstruction sp M,Ob}, $g'_*(E)\in\operatorname{HS}_k^{B'}(T)$ is a lift of $g_*(D)\in\operatorname{HS}_k^{B}(R)$. We have the following commutative diagram
\[\begin{tikzcd}
	{\varphi_E:} & T & {A'\otimes_k T} & {A'\otimes_k R} \\
	{\varphi_{g'_*(E)}:} & T & {B'\otimes_k T} & {B'\otimes_k R}
	\arrow["{E|_T}", from=1-2, to=1-3]
	\arrow[equals, from=1-2, to=2-2]
	\arrow[from=1-3, to=1-4]
	\arrow["{g'\otimes_k \operatorname{id}_T}"', from=1-3, to=2-3]
	\arrow["{g'\otimes_k \operatorname{id}_R}"', from=1-4, to=2-4]
	\arrow["{g'_*(E)|_T}"', from=2-2, to=2-3]
	\arrow[from=2-3, to=2-4]
\end{tikzcd}\]
This diagram restricts to
\[\begin{tikzcd}
	{f_E:} & {I/I^2} & {(k\cdot t)\otimes_k R} & R \\
	{f_{g'_*(E)}:} & {I/I^2} & {(k\cdot s)\otimes_k R} & R
	\arrow[from=1-2, to=1-3]
	\arrow[equals, from=1-2, to=2-2]
	\arrow["\simeq", from=1-3, to=1-4]
	\arrow["{a\cdot}"', from=1-3, to=2-3]
	\arrow["{a\cdot}"', from=1-4, to=2-4]
	\arrow[from=2-2, to=2-3]
	\arrow["\simeq"', from=2-3, to=2-4]
\end{tikzcd}\]
and we have $a\cdot f_E=f_{g'_*(E)}$. This proves the assertion.
\end{proof}

\begin{definition}\label{i Phi}
    Let $\Phi: \ 0\to k\cdot t\longrightarrow A' \overset{
    }{\longrightarrow}A \to 0$ be a small extension. Let $S\subset \mathbb{Z}$ be the set consisting of positive integers $n$ for which there exists $x\in \mathfrak{m}_{A'}$ with $x^n\in k\cdot t$ and $x^n\neq 0$. Define $i(\Phi)$ as the largest integer $i$ such that $p^i$ divides $n$ for some $n\in S$. In other words, it is the largest integer $i$ for which there exists $x\in \mathfrak{m}_{A'}$ with $x^{p^i}=t$.
\end{definition}
\begin{remark}
    For $x\in \mathfrak{m}_{A'}$ such that $x^n\in k\cdot t$ and $x^n\neq 0$, we have the following commutative diagram
\[\begin{tikzcd}
	0 & {k\cdot \varepsilon^n} & {k[\varepsilon]/\varepsilon^{n+1}} & {k[\varepsilon]/\varepsilon^{n}} & 0 \\
	0 & {k\cdot t} & {A'} & A & 0
	\arrow[from=1-1, to=1-2]
	\arrow[from=1-2, to=1-3]
	\arrow["{a\cdot}"', from=1-2, to=2-2]
	\arrow[from=1-3, to=1-4]
	\arrow["f", from=1-3, to=2-3]
	\arrow[from=1-4, to=1-5]
	\arrow[from=1-4, to=2-4]
	\arrow[from=2-1, to=2-2]
	\arrow[from=2-2, to=2-3]
	\arrow[from=2-3, to=2-4]
	\arrow[from=2-4, to=2-5]
\end{tikzcd}\]
where $f$ is defined by $\varepsilon \mapsto x$ and $a\neq 0$. As $k= k^p=k^{p^2}=\dots,$ we have
$$\operatorname{Ob}_R^n=a\cdot\operatorname{Ob}_R^n\subset \operatorname{Ob}_R^{\Phi}.$$ In particular, for $q=p^{i(\Phi)}$, we have $$\operatorname{Ob}_R^{q}\subset \operatorname{Ob}_R^{\Phi}.$$
\end{remark}

\begin{definition}\label{r def}
    Let $A$ be an object of $\mathcal{A}$ and $x\in \mathfrak{m}_A$. Define $r(A,x)$ as the largest integer $r$ such that $x^r\neq0$. If $x=0$, set $r(A,x)=0.$
\end{definition}

\begin{definition}\label{bracket def}
    Let $A$ be an object of $\mathcal{A}$, $x\in \mathfrak{m}_A$ and $R$ be a $k$-algebra. Let $n\ge r(A,x)$ be an integer and $f: k[\varepsilon]/\varepsilon^{n+1}\to A$ be given by $\varepsilon \mapsto x$. Define
    $$\langle x \rangle:\operatorname{HS}_k^{n}(R)\to \operatorname{HS}_k^A(R)$$ by
    $$D\langle x \rangle:=f_*(D)\in \operatorname{HS}_k^A(R).$$
    Suppose that, in addition, $\Phi: \ 0\to k\cdot x\longrightarrow A \overset{
    }{\longrightarrow}A/x \to 0$ is a small extension and $d\in \operatorname{Der}_k(R,R).$ Using the fact that $\widetilde{d}:=(\operatorname{id}, d)$ is an element of $\operatorname{HS}_k^1(R)$, we define
    $$d\langle x \rangle:=\widetilde{d}\langle x \rangle\in \operatorname{HS}_k^{A}(R).$$
\end{definition}
\begin{remark}
    This notation is compatible with that of Definition \ref{Hasse Schmidt}: If $m\ge 2$, $A=k[\varepsilon]/\varepsilon^{mn+1}$ and $x=\varepsilon^m$, then we can identify $D\langle \varepsilon^m \rangle$ and $D[m]$.
\end{remark}

\begin{lemma}\label{bracket lem}
Let $\Phi: \ 0\to k\cdot t\longrightarrow A' \overset{
    f}{\longrightarrow}A \to 0$ be a small extension and let $R$ be a $k$-algebra. Then, for $D\in \operatorname{HS}_k^{A'}(R)$, we have $D=d\langle t \rangle$ for some $  d\in \operatorname{Der}_k(R,R)$ if and only if $f_*(D)=\operatorname{id}$.
    Moreover, $d\langle t \rangle$ is in the center of $\operatorname{HS}_k^{A'}(R).$
\end{lemma}

\begin{proof}Let $D\in \operatorname{HS}_k^{A'}(R)$ and suppose that $f_*(D)=\operatorname{id}$. There exists a map $d:R\to R$ such that
$$D|_R(x)=1\otimes x + t\otimes d(x),\ x\in R.$$ We can see that the axioms of a $k$-derivation hold for $d$ and we have $D=d\langle t \rangle$. Conversely, for every $d\in \operatorname{Der}_k(R,R)$, we have $f(d\langle t \rangle)=d\langle 0\rangle=\operatorname{id}$. Now, let us show that $d\langle t \rangle$ is in the center of $\operatorname{HS}_k^{A'}(R).$ Let us take any $D'\in \operatorname{HS}_k^{A'}(R)$. It suffices to check that $D'\circ d\langle t \rangle=d\langle t \rangle\circ D'$ over $1\otimes R$ since both sides are $A'$-linear. For $x\in R$,
\begin{align*}D'\circ d\langle t \rangle(1\otimes x)&=D'(1\otimes x+t\otimes d(x))\\ &=D'(1\otimes x)+t\cdot D'(1\otimes d(x))\\ &=D'(1\otimes x)+t\otimes d(x).\end{align*}
The last equality is because $D'$ is the identity up to $\mathfrak{m}_{A'}$. On the other hand, as $d\langle t \rangle$ is the identity over $\mathfrak{m}_{A'}\otimes R$,
$$d\langle t \rangle(D'(1\otimes x)-1\otimes x)=D'(1\otimes x)- 1\otimes x$$
and
$$d\langle t \rangle(D'(1\otimes x))=D'(1\otimes x)+t\otimes d(x)$$
so that $D'\circ d\langle t \rangle=d\langle t \rangle\circ D'$.
\end{proof}

\begin{definition}\label{Ae lemma}
    Let $n\ge 0$, and let $e=(e_1,\dots, e_n)$ be a sequence of positive integers and $a=(a_1,\dots,a_n)$ be a sequence in $k$. We define objects $A(e)$ and $A(e,a)$ in $\mathcal{A}$ as follows: Let $k[X_1,\dots,X_n,T]$ be a polynomial ring in $n+1$ variables and let
    $$I_1:=(X_1,\dots,X_n), \ I_2:=(X_1^{e_1},\dots,X_n^{e_n})$$ be ideals of $k[X_1,\dots,X_n,T]$ and let $$f_i:=X_i^{e_i}-a_iT\in k[X_1,\dots,X_n,T].$$ Define
    $$A(e):=k[X_1,\dots,X_n,T]/(I_2,T),$$
    $$A(e,a):=k[X_1,\dots,X_n,T]/(I_1I_2,T^2,TI_1, f_1,\dots ,f_n).$$
\end{definition}

\begin{remark}\label{basis rem}
    Let $J_1=(X_1,\dots,X_n),J_2=(X_1^{e_1},\dots, X_n^{e_n})$ be ideals of $P:=k[X_1,\dots,X_n].$ Then, we have $P/J_2 \simeq A(e)$ and $\{X^\lambda :=X_1^{\lambda_1} \dots X_n^{\lambda_n} | \ 0\le \lambda_i < e_i\}$ is a $k$-linear basis of $A(e)$. We have the following push-out diagram
\[\begin{tikzcd}
	0 & {k\cdot X_1^{e_1}+\dots+k\cdot X_n^{e_n}} & {P/J_1J_2} & {P/J_2} & 0 \\
	0 & {k\cdot T} & {A(e,a)} & {A(e)} & 0
	\arrow[from=1-1, to=1-2]
	\arrow[from=1-2, to=1-3]
	\arrow["g"', from=1-2, to=2-2]
	\arrow[from=1-3, to=1-4]
	\arrow[from=1-3, to=2-3]
	\arrow[from=1-4, to=1-5]
	\arrow["\simeq"', from=1-4, to=2-4]
	\arrow[from=2-1, to=2-2]
	\arrow[from=2-2, to=2-3]
	\arrow[from=2-3, to=2-4]
	\arrow[from=2-4, to=2-5]
\end{tikzcd}\]
whose rows are exact, where $g$ is defined by $X_i^{e_i}\mapsto a_iT$. In particular, we have $T\neq 0$ in $A(e,a)$. It follows that $\{T \} \cup \{X^\lambda | \ 0\le \lambda_i < e_i\}$ forms a $k$-linear basis of $A(e,a).$
\end{remark}

\begin{definition} \label{cube def}Let $n\ge1$ and $e=(e_1,\dots, e_n)$ be a sequence of positive integers.  Define 
$$C=C(e):=\{(\lambda_1,\dots ,\lambda_n)\in \mathbb{Z}^n|0\le \lambda_i<e_i \mathrm{\ for \ each \ } i\}$$
and regard $C\subset \mathbb{R}^n$. Let $V=V_n\subset \mathbb{R}^n$ be the hyperplane defined by $x_n=0$. Define a sequence $l_1,\dots, l_N$ of lines (one-dimensional $\mathbb{R}$-vector space) in $\mathbb{R}^n$ to be such that
    \begin{enumerate}
    \item $C=(C\cap V) \cup \bigcup_{i=1}^N(C\cap l_i),$
    \item $l_i\neq l_j$ if $i\neq j$,
    \item  for each $i$, we have $\sharp (C\cap l_i)>1$  (in other words, $C\cap l_i$ contains points other than the origin),
    \item for each $i$, $l_i$ is not contained in $V$ and
    \item if $l_i$ and the hyperplane $\{x_n=1\}$ intersects at $P_i:=(P_i^1,\dots,P_i^{n-1},1)$, then the $P_i$ are in the lexicographic order, i.e., if $1\le i<j\le N$, then there exists $1\le n_0 \le n-1$ such that $P_i^1=P_j^1,\dots,P_i^{n_0-1}=P_j^{n_0-1}$ and $P_i^{n_0}<P_j^{n_0}$.
    \end{enumerate}
Note that the above conditions determine the $l_i$ uniquely. Define $C_0:=C$ and $C_0\supset C_1\supset \dots \supset C_{N}=(C\cap V)$ by $C_{i+1}=(C_i\setminus l_{i+1})\cup \{0\}$. Define $A_i(e)\subset A(e)$ be the sub-$k$-space whose basis is $\{X^\lambda=X_1^{\lambda_1}\dots X_n^{\lambda_n}|\lambda=(\lambda_1,\dots \lambda_n)\in C_i\}.$

\end{definition}

\begin{lemma}\label{subring lemma} In the situation of Definition \ref{cube def}, 
    for every $0\le i\le N$, $A_i(e)$ is a subring of $A(e)$ and in particular is an object of $\mathcal{A}$. 
\end{lemma}
\begin{proof}
    The case $i=N$ is trivial. Let us fix $0\le i< N$. Let $\mathbb{N}=\{0,1,\dots\}$ and regard $\mathbb{N}^n\subset \mathbb{R}^n$. Let $\Lambda_i$ be the subset of $\mathbb{N}^n$ consisting of $(\lambda_1,\dots,\lambda_n)$ such that $\lambda_n\neq 0 $ and $$(P_{i+1}^1,\dots, P_{i+1}^{n-1})\le (\frac{\lambda_1}{\lambda_n},\dots,\frac{\lambda_{n-1}}{\lambda_n}) \mathrm{\ in \ the \ lexicographic \ order. } $$
    If $(\lambda_1,\dots,\lambda_n),(\lambda_1',\dots,\lambda_n')\in \Lambda_i$, then we have $$(\lambda_1'',\dots,\lambda_n''):=(\lambda_1,\dots,\lambda_n)+(\lambda_1',\dots,\lambda_n')\in \Lambda_i$$ because $(\frac{\lambda_1''}{\lambda_n''},\dots,\frac{\lambda_{n-1}''}{\lambda_n''})$ lies on the line segment between $(\frac{\lambda_1}{\lambda_n},\dots,\frac{\lambda_{n-1}}{\lambda_n})$ and $(\frac{\lambda_1'}{\lambda_n'},\dots,\frac{\lambda_{n-1}'}{\lambda_n'})$. Thus, we have $\Lambda_i +\Lambda_i\subset \Lambda_i\subset \Lambda_i\cup (\mathbb{N}^n\cap V).$ We also have $(\mathbb{N}^n\cap V) +(\mathbb{N}^n\cap V)\subset \Lambda_i  \cup (\mathbb{N}^n\cap V)$ and $\Lambda_i+(\mathbb{N}^n\cap V)\subset \Lambda_i  \cup (\mathbb{N}^n\cap V)$. This implies that the sub-$k$-space $R_i$ of $k[X_1,\dots,X_n]$ spanned by $X^\lambda=X_1^{\lambda_1}\dots X_n^{\lambda_n},\lambda=(\lambda_1,\dots \lambda_n)\in \Lambda_i \cup (\mathbb{N}^n\cap V) $ is a subring. The image of $R_i$ under $k[X_1,\dots,X_n]\twoheadrightarrow A(e)$ is $A_i(e)$ and this is also a subring of $A(e)$.
\end{proof}

\begin{lemma}\label{ideal lem} In the situation of Definition \ref{cube def}, let us fix $0\le i< N$ and let $B_i$ be the sub-$k$-space of $A_i(e)$ whose basis is $\{X^\lambda=X_1^{\lambda_1}\dots X_n^{\lambda_n}|\lambda=(\lambda_1,\dots \lambda_n)\in (C_i\setminus l_{i+1})\}$. Then $B_i$ is an ideal of $A_i(e)$.
\end{lemma}
\begin{proof}
    In the notation of Lemma \ref{subring lemma}, if $$(\lambda_1,\dots,\lambda_n),(\lambda_1',\dots,\lambda_n')\in \Lambda_i \cup (\mathbb{N}^n\cap V)$$ and $(\lambda_1,\dots,\lambda_n)\not\in l_{i+1}$ , then $$(\lambda_1'',\dots,\lambda_n''):=(\lambda_1,\dots,\lambda_n)+(\lambda_1',\dots,\lambda_n')\not\in l_{i+1}.$$
    Thus, the sub-$k$-space $\widetilde{B}_i$ of $R_i$ spanned by $\{X^\lambda|\lambda\in (\Lambda_i \cup (\mathbb{N}^n\cap V) )\setminus l_{i+1}\}$ is an ideal of $R_i$. The image $B_i$ of $\widetilde{B}_i$ under $R_i\twoheadrightarrow A_i(e)$ is also an ideal.
\end{proof}

\begin{lemma}\label{line lemma}In the situation of Definition \ref{cube def}, let us fix $0\le i< N$. Let $\lambda(i)$ denote the element closest to the origin in $ (C_i\cap l_{i+1})\setminus \{0\}$. Let $c(i)=r(A_i(e),X^{\lambda(i)})$ denote the largest integer such that $c(i)\cdot\lambda(i)\in C$. Then we have an isomorphism
$$A_i(e)/B_i\simeq k[\varepsilon]/\varepsilon^{c(i)+1},\ X^{\lambda(i)}\mapsto \varepsilon.$$ 
\end{lemma}

\begin{proof} First, note that we have a disjoint union decomposition $$C_i=\{0\}\cup ((C\cap V) \setminus \{0\})\cup \bigcup_{j=i+1}^N(C\cap l_j)$$
by induction on $i$. In particular, $ (C_i\cap l_{i+1})\setminus \{0\}= (C\cap l_{i+1})\setminus \{0\}$ is nonempty and $\lambda(i)$ is well-defined. As $C_i\cap l_{i+1}=\{0, \lambda(i),\dots,c(i)\cdot \lambda(i)\},$
    we have an isomorphism $A_i(e)/B_i\simeq k[\varepsilon]/\varepsilon^{c(i)+1}$ given by $X^{\lambda(i)}\simeq \varepsilon.$ 
\end{proof}

Let $\sigma$ be a permutation of $\{1,\dots,n\}$ and let $e_\sigma=(e_{\sigma(1)},\dots,e_{\sigma(n)}),\ a_\sigma=(a_{\sigma(1)},\dots,a_{\sigma(n)}).$ Then, we have a natural isomorphism $A(e,a)\simeq A(e_\sigma,a_\sigma)$ given by $X_i\mapsto X_{\sigma^{-1}(i)}.$

Recall the definitions of $r(A, y)$ and $D \langle y \rangle$ from Definition \ref{r def}, \ref{bracket def}.
\begin{lemma}\label{iota lemma}
In the situation of Definition \ref{cube def}, after permuting $e$ and $a$ simultaneously if necessary, the following statement holds: 

Let $e'=(e_1,\dots e_{n-1}),a'=(a_1,\dots,a_{n-1})$. Consider the following diagram
\[\begin{tikzcd}
	0 & {k\cdot T} & {A(e',a')} & {A(e')} & 0 \\
	0 & {k\cdot T} & {A(e,a)} & {A(e)} & 0
	\arrow[from=1-1, to=1-2]
	\arrow[from=1-2, to=1-3]
	\arrow[equals, from=1-2, to=2-2]
	\arrow["{\pi'}", from=1-3, to=1-4]
	\arrow["\iota", hook, from=1-3, to=2-3]
	\arrow[from=1-4, to=1-5]
	\arrow[hook, from=1-4, to=2-4]
	\arrow[from=2-1, to=2-2]
	\arrow[from=2-2, to=2-3]
	\arrow["\pi", from=2-3, to=2-4]
	\arrow[from=2-4, to=2-5]
\end{tikzcd}\]
where $\iota(X_i)=X_i,\ \iota(T)=T$. Let $R$ be a $k$-algebra essentially of finite type and $D\in \operatorname{HS}_k^{A(e,a)}(R)$. Then, there exist $y_1,\dots , y_{m}\in \mathfrak{m}_{A(e,a)}$ and $D^1\in \operatorname{HS}_k^{r(A(e,a),y_1)}(R) ,$ $\dots ,$  $D^m \in \operatorname{HS}_k^{r(A(e,a),y_m)}(R)$ such that for
$$E:=D^{m}\langle y_m \rangle\circ \dots \circ D^1\langle y_1 \rangle\circ D,$$
the image of $E|_R$ is contained $A(e',a')\otimes_k R$ (which is equivalent to say that there exists $F\in \operatorname{HS}_k^{A(e',a')}(R)$ such that $\iota_*(F)=E$, by Remark \ref{subring induce}).
\end{lemma}

\begin{proof}
First, for every $1\le i\le N$, let $A_i(e,a):=\pi^{-1}(A_i(e))$. We have the following commutative diagram
\[\begin{tikzcd}
	0 & {k\cdot T} & {A_i(e,a)} & {A_i(e)} & 0 \\
	0 & {k\cdot T} & {A(e,a)} & {A(e)} & 0
	\arrow[from=1-1, to=1-2]
	\arrow[from=1-2, to=1-3]
	\arrow[equals, from=1-2, to=2-2]
	\arrow["{\pi_i}", from=1-3, to=1-4]
	\arrow["{\iota_i}", hook, from=1-3, to=2-3]
	\arrow[from=1-4, to=1-5]
	\arrow[hook, from=1-4, to=2-4]
	\arrow[from=2-1, to=2-2]
	\arrow[from=2-2, to=2-3]
	\arrow["\pi", from=2-3, to=2-4]
	\arrow[from=2-4, to=2-5].
\end{tikzcd}\]

Note that the set $\{T\}\cup \{X^\lambda|\lambda\in C_i\}$ is a $k$-basis for $A_i(e,a)$ and for $i=N$, $\pi_N$ is equal to $\pi'$. For $D\in \operatorname{HS}_k^{A(e,a)}(R)$, let us write $D_\lambda:=D_{X^\lambda}$ where $\lambda \in C$ (see Definition \ref{basis def} and Remark \ref{basis rem}.) We show that, by induction on $i=1,2,\dots$, there exists $y_1,\dots , y_{m}\in \mathfrak{m}_{A(e,a)}$ and $D^1\in \operatorname{HS}_k^{r(A(e,a),y_1)}(R),\dots, D^m\in \operatorname{HS}_k^{r(A(e,a),y_m)}(R)$ such that for
$$E^i=D^{m}\langle y_m \rangle\circ \dots \circ D^1\langle y_1 \rangle\circ D,$$  we have $\{\lambda\in C | E^i_\lambda\neq 0\}\subset C_i$ (which is equivalent to say that there exists $F^i\in \operatorname{HS}_k^{A_i(e,a)}(R)$ such that $(\iota_i)_*(F^i)=E^i$). 

We treat the cases $i=1$ and $i=2,\dots,N$ separately.

For the case $i=1$, let $w=(0,0,\dots,0,1)\in C$. We would like $$E^1=D^{m}\langle y_m \rangle\circ \dots \circ D^1\langle y_1 \rangle\circ D$$
to satisfy $E^1_w=E^1_{2w}=\dots E^1_{(e_n-1)w}=0$. 

First, suppose that one of the $a_i$ is zero. By rearranging, we can suppose that $a_n=0$. Consider the restriction of $D$ under $A(e,a)\to A_0(e)\to A_0(e)/B_0$. By Lemma \ref{line lemma}, we have
$$D^1:=(\operatorname{id}, D_w,\dots, D_{(e_n-1)w})\in \operatorname{HS}_k^{e_n-1}(R). $$
By the assumption, we have $r(A(e,a),X_n)=e_n-1$. We define $E^1$ to be
$$(D^1)^{-1}\langle X_n \rangle\circ D\in \operatorname{HS}_k^{A(e,a)}(R),$$
which satisfies $E_w=E_{2w}=\dots=E_{(e_n-1)w}=0.$
Next, suppose that none of the $a_i$ is zero. By rearranging again we can suppose that, if $j$ is the largest integer such that $p^j$ divides $e_n$, then $p^j$ also divides $e_{1}$. Let us write $e_n=p^jg, e_1=p^jh.$ We have that, in $A(e,a)$, $$(X_n^g-\frac{\sqrt[p^j]{a_n}}{\sqrt[p^j]{a_1}}X_1^h)^{p^j}=X_n^{e_n}-\frac{a_n}{a_1}X_1^{e_1}=a_nT-a_nT=0$$
so that $r(A(e,a),X_n^g-\frac{\sqrt[p^j]{a_n}}{\sqrt[p^j]{a_1}}X_1^h)< p^j.$ As in the first case, we have that
$$D^1:=(\operatorname{id}, D_w,\dots, D_{(e_n-1)w})\in \operatorname{HS}_k^{e_n-1}(R). $$
We have seen in Section 1 that there exists
$D^2\in \operatorname{HS}_k^{p^j-1}(R)$ such that $(D^1)^{-1} , D^2[g]\in \operatorname{HS}_k^{e_n-1}(R)$ have the same obstruction. Let
$$\widetilde{E}=D^2\langle X_n^g-\frac{\sqrt[p^j]{a_n}}{\sqrt[p^j]{a_1}}X_1^h\rangle\circ D\in \operatorname{HS}_k^{A(e,a)}(R).$$
We can compute
$$\widetilde{E}^1:=(\operatorname{id}, \widetilde{E}_w,\dots, \widetilde{E}_{(e_n-1)w})\in \operatorname{HS}_k^{e_n-1}(R)$$
by passing to $A(e,a)\to A_0(e)\to A_0(e)/B_0\simeq k[\varepsilon]/\varepsilon^{e_n}.$ Over $k[\varepsilon]/\varepsilon^{e_n}$, we have
$$\widetilde{E}^1=D^2\langle \varepsilon^g-\frac{\sqrt[p^j]{a_n}}{\sqrt[p^j]{a_1}}0^h\rangle\circ D^1=D^2\langle \varepsilon^g\rangle \circ D^1=D^2[ g ]\circ D^1.$$

Hence, we have $\operatorname{ob}_{e_n}(\widetilde{E}^1)=0$, and this lifts to $D^3\in \operatorname{HS}_k^{e_n}(R)$. Now we set $E^1$ by
$$E^1=(D^3)^{-1}\langle X_n \rangle\circ \widetilde{E}\in \operatorname{HS}_k^{A(e,a)}(R),$$
which satisfies $E^1_w=E^1_{2w}=\dots=E^1_{(e_n-1)w}=0.$ This completes the induction step for $i=1$.

Now suppose that for $1\le i\le N-1$, we have reduced to the case that
$$E^i=D^{m}\langle y_m \rangle\circ \dots \circ D^1\langle y_1 \rangle\circ D$$ satisfies $\{\lambda\in C | E^i_\lambda\neq 0\}\subset C_i$, which is equivalent to saying that there exists $F^i\in \operatorname{HS}_k^{A_i(e,a)}(R)$ such that $(\iota_i)_*(F^i)=E^i$.

We use the notation introduced in Lemma \ref{line lemma}. By the lemma, we have
$$G:=(\operatorname{id}, F^i_{1\cdot \lambda(i)},\dots, F^i_{c(i)\cdot \lambda(i)})\in \operatorname{HS}_k^{c(i)}(R). $$

We note that $$c(i)=r(A(e),X^{\lambda(i)})=r(A(e,a),X^{\lambda(i)})=r(A_i(e,a),X^{\lambda(i)}).$$ The first equality follows from the definition, and the third holds because $A_i(e,a)$ is a subring of $A(e,a)$. For the second equality, suppose $X^{j\cdot \lambda(i)}=0$ in $A(e)$ for some $j\ge 1$. This implies that, for some $1\le n_0 \le n$, we have $j\cdot \lambda_{n_0}(i)\ge e_{n_0}$ where $\lambda(i)=(\lambda_1(i),\dots,\lambda_{n}(i))$. Because $i\ge 1$, at least two entries of $\lambda(i)$ are nonzero. This implies that $X^{j\cdot \lambda(i)}\in I_1I_2$ in the notation of Definition \ref{Ae lemma}. Thus we also have $X^{j\cdot \lambda(i)}=0$ in $A(e,a)$. By the definition of $r$, we have $r(A(e),X^{\lambda(i)})\ge r(A(e,a),X^{\lambda(i)})$. The converse inequality is clear and we have $r(A(e),X^{\lambda(i)})=r(A(e,a),X^{\lambda(i)})$.

As $c(i)=r(A_i(e,a),X^{\lambda(i)})$, we can define
$$\widetilde{F}:=G^{-1}\langle X^{\lambda (i)}\rangle \circ F^i\in \operatorname{HS}_k^{A_i(e,a)}(R).$$
By passing to $A_i(e,a)\to A_i(e)\to A_i(e)/B_i$ as in the proof of Lemma \ref{line lemma}, we see that $\widetilde{F}_\lambda=0$ for $\lambda\in (C\cap l_{i+1})\setminus \{0\}$. Hence, if we set $E^{i+1}$ as
$$E^{i+1}:=(\iota_i)_*(\widetilde{F})=G^{-1}\langle X^{\lambda (i)}\rangle \circ E^i\in \operatorname{HS}_k^{A(e,a)}(R),$$
we have $\{\lambda\in C | E^{i+1}_\lambda \neq 0\} \subset C_{i+1}$. This step completes the proof.\end{proof}

\begin{proposition}\label{decomposition prop}
    Let $R$ be a $k$-algebra essentially of finite type. Let $n\ge 0$, $e=(e_1,\dots, e_n)$ be a sequence of positive integers and let $a=(a_1,\dots,a_n)$ be a sequence in $k$. For any $D\in \operatorname{HS}_k^{A(e,a)}(R)$, there exists a sequence $y_1,\dots,y_m\in \mathfrak{m}_{A(a,e)}$ and $D^i\in \operatorname{HS}_k^{r(A(e,a),y_i)}(R) $ for every $1\le i \le m$ such that 
    $$D=D^1\langle y_1 \rangle\circ \dots \circ D^m\langle y_m \rangle.$$
\end{proposition}

\begin{proof}
    We proceed by induction on $n$. The case $n=0$ is trivial, since $A(e,a)\simeq k[\varepsilon]/\varepsilon^2$. If $n\ge 1$, wa may use the previous lemma. Rearrange $(e_1,\dots, e_n)$ and $(a_1,\dots,a_n)$ if necessary and let $e'=(e_1,\dots e_{n-1}),a'=(a_1,\dots,a_{n-1})$. We have the following diagram
\[\begin{tikzcd}
	0 & {k\cdot T} & {A(e',a')} & {A(e')} & 0 \\
	0 & {k\cdot T} & {A(e,a)} & {A(e)} & 0
	\arrow[from=1-1, to=1-2]
	\arrow[from=1-2, to=1-3]
	\arrow[equals, from=1-2, to=2-2]
	\arrow[from=1-3, to=1-4]
	\arrow["\iota", hook, from=1-3, to=2-3]
	\arrow[from=1-4, to=1-5]
	\arrow[hook, from=1-4, to=2-4]
	\arrow[from=2-1, to=2-2]
	\arrow[from=2-2, to=2-3]
	\arrow[from=2-3, to=2-4]
	\arrow[from=2-4, to=2-5].
\end{tikzcd}\]
We can find $y_1,\dots , y_{m'}\in \mathfrak{m}_{A(e,a)}$ and $D^i\in \operatorname{HS}_k^{r(A(e,a),y_i)}(R)$ such that for
$$E:=D^{m'}\langle y_{m'}\rangle \circ \dots \circ D^1\langle y_1 \rangle\circ D,$$
the image of $E|_R$ is contained $A(e',a')\otimes_k R$. By Remark \ref{subring induce}, there exists $F\in \operatorname{HS}_k^{A(e',a')}(R)$ such that $\iota_*(F)=E$. By the induction hypothesis, there exist $z_1,\dots,z_{m''}\in \mathfrak{m}_{A(e',a')}$ and $F^i\in \operatorname{HS}_k^{r(A(e',a'),z_i)}(R)$ such that
$$F=F^1\langle z_1 \rangle \circ \dots \circ F^{m''}\langle z_{m''}\rangle.$$
Then, we have that
$$D=(D^1)^{-1}\langle y_1 \rangle\circ \dots \circ (D^{m'})^{-1}\langle y_{m'}\rangle \circ F^1\langle z_1\rangle \circ \dots \circ F^{m''}\langle z_{m''}\rangle  $$
and this completes the induction step.
\end{proof}

\begin{theorem}\label{decomposition thm}
    Let $R$ be a $k$-algebra essentially of finite type. Let $A$ be an object of $\mathcal{A}$. For any $D\in \operatorname{HS}_k^{A}(R)$, there exists a sequence $y_1,\dots,y_m\in \mathfrak{m}_{A}$ and $D^i\in \operatorname{HS}_k^{r(A,y_i)}(R) $ for every $1\le i \le m$ such that 
    $$D=D^1\langle y_1 \rangle\circ \dots \circ D^m\langle y_m \rangle.$$
    (See Definitions \ref{r def} and \ref{bracket def} for the symbols.)
\end{theorem}

\begin{proof}
By induction of the length of $A$. If the length is one, then $A=k$ and this case is trivial. To proceed with the induction step, consider a small extension
$$\Psi: 0\to k\cdot t\to A'\overset{\pi}{\to} A\to 0.$$
It suffices to show that, if the theorem is true for $A$, then so for $A'$. Let $D\in \operatorname{HS}_k^{A'}(R)$. By the induction hypothesis, there exists a decomposition
$$\pi_*(D)=D^1\langle y_1 \rangle\circ \dots \circ D^m\langle y_m \rangle$$
for $y_1,\dots,y_m\in \mathfrak{m}_{A}$ and $D^i\in \operatorname{HS}_k^{r(A,y_i)}(R) $. 
Let $z_1,\dots,z_m$ be lifts of $y_i$ to $\mathfrak{m}_{A'}$. Let $e_i=r(A,y_i)+1$ for every $i=1,\dots, m$. Then, we can write $z_i^{e_i}=a_it,\ a_i\in k$ for every $i=1,\dots,m$. Let $e=(e_1,\dots,e_m),a=(a_1,\dots,a_m)$. We have the following commutative diagram
\[\begin{tikzcd}
	{\Phi:} & 0 & {k\cdot T} & {A(e,a)} & {A(e)} & 0 \\
	{\Psi:} & 0 & {k\cdot t} & {A'} & A & 0
	\arrow[from=1-2, to=1-3]
	\arrow[from=1-3, to=1-4]
	\arrow[from=1-3, to=2-3]
	\arrow[from=1-4, to=1-5]
	\arrow["f", from=1-4, to=2-4]
	\arrow[from=1-5, to=1-6]
	\arrow["g", from=1-5, to=2-5]
	\arrow[from=2-2, to=2-3]
	\arrow[from=2-3, to=2-4]
	\arrow[from=2-4, to=2-5]
	\arrow[from=2-5, to=2-6]
\end{tikzcd}\]
where $f(X_i)=z_i,\ f(T)=t$ and $g(X_i)=y_i$. For each $1\le i \le m$, we also have the commutative diagram
\[\begin{tikzcd}
	0 & {k\cdot\varepsilon^{e_i}} & {k[\varepsilon]/\varepsilon^{e_i+1}} & {k[\varepsilon]/\varepsilon^{e_i}} & 0 \\
	0 & {k\cdot T} & {A(e,a)} & {A(e)} & 0
	\arrow[from=1-1, to=1-2]
	\arrow[from=1-2, to=1-3]
	\arrow["{a_i\cdot}"', from=1-2, to=2-2]
	\arrow[from=1-3, to=1-4]
	\arrow["{f_i}"', from=1-3, to=2-3]
	\arrow[from=1-4, to=1-5]
	\arrow["{g_i}", from=1-4, to=2-4]
	\arrow[from=2-1, to=2-2]
	\arrow[from=2-2, to=2-3]
	\arrow[from=2-3, to=2-4]
	\arrow[from=2-4, to=2-5]
\end{tikzcd}\]
where $f_i(\varepsilon)=X_i, \ g_i(\varepsilon)=X_i$. Let $E^i:=(g_i)_*(D^i)\in \operatorname{HS}_k^{A(e)}(R)$. We have $D^i\langle y_i\rangle =g_*(E^i)$ in $\operatorname{HS}_k^{A}(R)$ because of the functoriality of $\operatorname{HS}_k^-(R)$. Thus, $E:=E^1\circ\dots\circ E^m$ is a lift of $\pi_*(D)$ to $A(e)$. We have
$$\operatorname{ob}_\Phi (E) =1\cdot \operatorname{ob}_\Phi (E)=\operatorname{ob}_\Psi (g_*(E))=\operatorname{ob}_\Psi (\pi_*(D)),$$
which is zero because $\pi_*(D)$ lifts to $D$ along $\Psi$. Thus, also $E$ lifts to $F\in \operatorname{HS}_k^{A(e,a)}(R)$. By Proposition \ref{decomposition prop}, $F$ decomposes as
$$F=F^1\langle u_1\rangle \circ\dots\circ F^{m'}\langle u_{m'}\rangle .$$ By construction, $D$ and 
$$f_*(F)=F^1\langle f(u_1)\rangle \circ\dots\circ F^{m'}\langle f(u_{m'})\rangle $$
agree modulo $t$. Hence, there exists $d\in \operatorname{Der}_k(R,R)$ such that
$$D=d\langle t \rangle\circ F^1\langle f(u_1)\rangle \circ\dots\circ F^{m'}\langle f(u_{m'})\rangle $$
and this completes the induction.
\end{proof}

\begin{theorem}\label{artin obst theorem}
    Let $R$ be a $k$-algebra essentially of finite type and $$\Phi: \ 0\to k\cdot t\longrightarrow A' \overset{
    }{\longrightarrow}A \to 0$$ be a small extension. Then, we have
    $$\operatorname{Ob}_R^{q}= \operatorname{Ob}_R^{\Phi}$$
    where $q=p^{i(\Phi)}$ (see Definition \ref{i Phi} for $i(\Phi)$).
    In particular, $\operatorname{Ob}_R^{\Phi}$ is a finite $R^{q}$-module and is compatible with localization in the sense stated in Proposition \ref{Ob module}. 
\end{theorem}
\begin{proof}
It suffices to show that $\operatorname{Ob}_R^{q}\supset  \operatorname{Ob}_R^{\Phi}.$ For every $D\in \operatorname{HS}_k^A(R)$, it decomposes as $D=D^1\langle y_1 \rangle\circ \dots \circ D^m\langle y_m \rangle$ where $y_1,\dots,y_m\in \mathfrak{m}_{A}$ and $D^i\in \operatorname{HS}_k^{r(A,y_i)}(R)$ for every $1\le i \le m$. For every $i$, there exists a commutative diagram

\[\begin{tikzcd}
	0 & {k\cdot \varepsilon^{r_i+1}} & {k[\varepsilon]/\varepsilon^{r_i+2}} & {k[\varepsilon]/\varepsilon^{r_i+1}} & 0 \\
	0 & {k\cdot t} & {A'} & A & 0
	\arrow[from=1-1, to=1-2]
	\arrow[from=1-2, to=1-3]
	\arrow["{a_i\cdot}"', from=1-2, to=2-2]
	\arrow[from=1-3, to=1-4]
	\arrow["f", from=1-3, to=2-3]
	\arrow[from=1-4, to=1-5]
	\arrow[from=1-4, to=2-4]
	\arrow[from=2-1, to=2-2]
	\arrow[from=2-2, to=2-3]
	\arrow[from=2-3, to=2-4]
	\arrow[from=2-4, to=2-5]
\end{tikzcd}\]
where $r_i=r(A,y_i)$, $a_i\in k$ and $f(\varepsilon)$ is a lift of $y_i$ to $A'$. We have
\begin{align*}
\operatorname{ob}_\Phi(D)&=\operatorname{ob}_\Phi(D^1\langle y_1 \rangle)+\dots+\operatorname{ob}_\Phi(D^m\langle y_m \rangle) \\ &=a_1 \cdot\operatorname{ob}_{r_1+1}(D^1) +\dots + a_m \cdot\operatorname{ob}_{r_m+1}(D^m).
\end{align*}
In the above sum, if $a_i\neq 0$, then $r_i+1 \in S$ in the notation of Definition \ref{i Phi} and, in particular, $\operatorname{ob}_{r_i+1}(D^i) \in \operatorname{Ob}_R^q$. Since $\operatorname{Ob}_R^q$ is closed under multiplication by $k^q=k$, the above sum lies in $\operatorname{Ob}_R^q$.\end{proof}

By Theorem \ref{artin obst theorem}, we obtain:

\begin{corollary}\label{Ob sheaf artin}
    Let $X$ be a scheme of finite type over $k$ and
    $$\Phi:0\to k\cdot t\longrightarrow A'\longrightarrow A\to 0$$
    be a small extension.
    Then there exists a subsheaf
    $$\operatorname{Ob}_X^{\Phi}\subset T_{X/k}^1$$
    such that for every affine open subscheme $\operatorname{Spec}R\subset X,$ we have
    $$\operatorname{Ob}_X^{\Phi}(\operatorname{Spec}R)=\operatorname{Ob}_R^\Phi\subset T_{R/k}^1.$$ In fact, we have $$\operatorname{Ob}_X^{\Phi}=\operatorname{Ob}_X^{q}$$
    where $q=p^{i(\Phi)}$ and it has the structure of a coherent $\mathcal{O}_X$-module as stated in Theorem \ref{Ob sheaf}.
    
\end{corollary}

\begin{definition}
    Let $R$ be a $k$-algebra essentially of finite type and
    $$\Phi:0\to k\cdot t\longrightarrow A'\longrightarrow A\to 0$$
    be a small extension. Define
    $$\operatorname{IHS}^\Phi(R)\subset \operatorname{HS}_k^A(R)$$
    to be the kernel of $\operatorname{ob}_\Phi$. It is precisely the image of the map $\operatorname{HS}_k^{A'}(R)\to\operatorname{HS}_k^A(R)$.
    Let $X$ be a scheme of finite type over $k$. Define
    $$\underline{\operatorname{HS}}_k^A$$
    to be the sheaf of groups on $X$, where $\underline{\operatorname{HS}}_k^A(U)=\operatorname{HS}_k^A(R)$ for every affine open subscheme $U=\operatorname{Spec}R\subset X$ with canonical restriction maps. Define
    $$\operatorname{ob}_\Phi: \underline{\operatorname{HS}}_k^A \to \operatorname{Ob}_X^\Phi$$ by gluing the local maps over affine open subschemes. Define $$\underline{\operatorname{IHS}}^\Phi\subset \underline{\operatorname{HS}}_k^A$$ to be the kernel of $\operatorname{ob}_\Phi$. We have $\underline{\operatorname{IHS}}^\Phi(U)=\operatorname{IHS}^\Phi(R)$ for every affine open subscheme $U=\operatorname{Spec}R\subset X$.
\end{definition}
\begin{theorem}\label{ihs thm}
 Let $X$ be a scheme of finite type over $k$ and
    $$\Phi:0\to k\cdot t\longrightarrow A'\longrightarrow A\to 0$$
    be a small extension. We have the following exact sequences of sheaves of groups on $X$:
    $$1\to \underline{\operatorname{IHS}}^\Phi\longrightarrow \underline{\operatorname{HS}}_k^A\longrightarrow \operatorname{Ob}_X^\Phi \to 0,$$
    $$0\to \operatorname{Der}_k(\mathcal{O}_X) \overset{d\mapsto d\langle t\rangle}{\longrightarrow} \underline{\operatorname{HS}}_k^{A'}\longrightarrow \underline{\operatorname{IHS}}^\Phi\to 1.$$
For any affine open subscheme $U\subset X$, the sections of the above sequences over $U$ are also exact.
    
\end{theorem}

\section{Preliminaries on locally trivial deformations}

\begin{definition}
    Let $A$ be an object of $\mathcal{A}$ and let $X_0$ be an algebraic scheme over $k$ (i.e. separated and of finite type over $k$). A cartesian diagram of schemes
\[\begin{tikzcd}
	{X_0} & X \\
	{\operatorname{Spec}k=\operatorname{Spec}A/\mathfrak{m}_A} & {\operatorname{Spec}A}
	\arrow[from=1-1, to=1-2]
	\arrow[from=1-1, to=2-1]
	\arrow[from=1-2, to=2-2]
	\arrow[from=2-1, to=2-2]
\end{tikzcd}\]
is called a \emph{(formal) deformation} of $X_0$ over $A$ if $X$ is flat over $\operatorname{Spec}A$.
\end{definition}

Note that $X$ and $X_0$ have the same underlying topological space. By abuse of notation, we simply say that $X$ is a deformation of $X_0$ over $A$.
\begin{definition}
     Let $A$ be an object of $\mathcal{A}$ and let $X_0$ be an algebraic scheme. Let $X$ and $Y$ be deformations of $X_0$ over $A$. An \emph{isomorphism} between $X$ and $Y$ is a commutative diagram of schemes
\[\begin{tikzcd}
	X & Y \\
	{\operatorname{Spec}A} & {\operatorname{Spec}A}
	\arrow[from=1-1, to=1-2]
	\arrow[from=1-1, to=2-1]
	\arrow[from=1-2, to=2-2]
	\arrow[equals, from=2-2, to=2-1]
\end{tikzcd}\]
such that, when restricted along $A\to A/\mathfrak{m}_A=k$, we obtain
\[\begin{tikzcd}
	{X_0} & {X_0} \\
	{\operatorname{Spec}k} & {\operatorname{Spec}k}
	\arrow[equals, from=1-1, to=1-2]
	\arrow[from=1-1, to=2-1]
	\arrow[from=1-2, to=2-2]
	\arrow[equals, from=2-2, to=2-1].
\end{tikzcd}\]
    
\end{definition}

\begin{definition}
    Let $A$ be an object of $\mathcal{A}$ and $X_0$ be an algebraic scheme. Then, a deformation $X$ is called \emph{locally trivial} if, for every $x\in X_0$, there exists an open neighborhood $x\in U\subset X_0$ such that the restriction $X|_U$ is isomorphic to the trivial deformation of $U$.
\end{definition}

\begin{definition}\label{A_n,  Def'_n, Aut_n} Let $A$ be an object of $\mathcal{A}$ and $X_0$ be an algebraic scheme. Then, the symbol $\operatorname{Def}_{X_0}'(A)$ denotes the set of locally trivial deformations of $X_0$ over $A$, modulo isomorphisms. It defines a covariant functor from $\mathcal{A}$ to the category of sets, which we call the {\em locally trivial deformation functor} of $X_0$.
\end{definition}

We now explain Schlessinger's conditions \cite[Theorem 2.11]{Sch}. Let $F:\mathcal{A}\to (\mathrm{sets})$ be a covariant functor (such an $F$ is called a {\emph{functor of artin rings}). We say that $F$ satisfies \emph{Schlessinger's condition $(H_1)$} if, for any cartesian diagram 
\[\begin{tikzcd}
	{\overline{A}=A'\times_AA''} & {A''} \\
	{A'} & A
	\arrow[from=1-1, to=1-2]
	\arrow[from=1-1, to=2-1]
	\arrow["g"', from=1-2, to=2-2]
	\arrow["f", from=2-1, to=2-2]
\end{tikzcd}\]
in $\mathcal{A}$, the induced canonical map
$$F(\overline{A})\to F(A')\times_{F(A)} F(A'')$$
is surjective whenever $f:A'\to A$ is a small extension. Here, we have $\overline{A}=\{(a,b)\in A'\times A''|f(a)=g(b)\}$ by definition. We say that $F$ satisfies \emph{Schlessinger's condition $(H_2)$} if the above map is bijective if $A''=k[\varepsilon]/\varepsilon^2, A=k$ and $A'\to A$ is a small extension.

In Section 6, We will construct an algebraic $k$-scheme $X_0$ such that $\operatorname{Def}'_{X_0}$ does not satisfy $(H_1)$.

\begin{lemma}
    Let $X_0$ be an algebraic scheme and let
\[\begin{tikzcd}
	{\overline{A}=A'\times_AA''} & {A''} \\
	{A'} & A
	\arrow[from=1-1, to=1-2]
	\arrow[from=1-1, to=2-1]
	\arrow["g"', from=1-2, to=2-2]
	\arrow["f", from=2-1, to=2-2]
\end{tikzcd}\]
be a cartesian diagram in $\mathcal{A}$. Suppose that $A=k, A''=k[\varepsilon]/\varepsilon^2$. Let $X'$ (resp. $X''$) be a locally trivial deformation of $X_0$ over $A'$ (resp. over $A''$). Then the pair $(X',X'')$ lifts to a unique locally trivial deformation $\overline{X}$ over $\overline{A}$. In particular, $\operatorname{Def}_{X_0}'$ satisfies Schlessinger's condition $(H_2)$.
\end{lemma}

\begin{proof}
 The uniqueness of $\overline{X}$ follows from the fact that $\operatorname{Def}_{X_0}'$ is a subfunctor of the usual deformation functor $\operatorname{Def}_{X_0}$ and the fact that $\operatorname{Def}_{X_0}$ satisfies $(H_2)$ itself by \cite[Example 3.7]{Sch}. Let
$$\mathcal{O}_{\overline{X}}:=\mathcal{O}_{X'}\times_{\mathcal{O}_{X_0}} \mathcal{O}_{X''}$$
be the cartesian product of sheaves over the topological space $X_0$. This is a flat $\overline{A}$-module by \cite[Lemma 3.4]{Sch}. Let $U:=\operatorname{Spec}R$ be an affine open subscheme of $X_0$. By Theorem \ref{affine}, we have
$$\mathcal{O}_{X'}(U)\simeq R\otimes_kA',\ \mathcal{O}_{X''}(U)\simeq R\otimes_kA''.$$

These isomorphisms induce a diagram
\[\begin{tikzcd}
	R & { \mathcal{O}_{X''}(U)} \\
	{\mathcal{O}_{X'}(U)} & { \mathcal{O}_{X_0}(U)}
	\arrow[from=1-1, to=1-2]
	\arrow[from=1-1, to=2-1]
	\arrow[from=1-2, to=2-2]
	\arrow[from=2-1, to=2-2]
\end{tikzcd}\]
which is necessarily commutative because $R=\mathcal{O}_{X_0}(U)$. By the universal property of a cartesian product, we have a $k$-algebra homomorphism $R\to \mathcal{O}_{\overline{X}}(U)$ and also an $\overline{A}$-algebra homomorphism $R\otimes_k \overline{A}\to \mathcal{O}_{\overline{X}}(U)$. The latter is an isomorphism by \cite[Lemma 3.3]{Sch}. Thus, $\mathcal{O}_{\overline{X}}$ defines a locally trivial lift of the pair $(X',X'')$.\end{proof}

\begin{remark}\label{fiber action rmk}
    The above lemma implies that we have a natural action of $H^1(X_0,\operatorname{Der}_k(\mathcal{O}_{X_0}))$ on the fibers of $\operatorname{Def}'_{X_0}$ as noted in \cite[Remark 2.15]{Sch}. See also Remark \ref{fiber idetification rmk}.
\end{remark}

\section{Obstruction theory of locally trivial deformations}

The aim of this section is to prove Theorem \ref{lt main thm}, which identifies the obstructions to lifting a locally trivial deformation along a small extension and describes the structure of the corresponding fiber of the functor $\operatorname{Def}'_{X_0}$.

We now recall the definition of non-abelian \v{C}ech cohomology (cf. \cite[Definition 11.11]{gw}). Suppose that $G$ is a sheaf of groups on a topological space $S$ and let $\mathcal{U}=\{U_\alpha\}_\alpha $ be an open covering of $S$. We say that $(g_{\alpha\beta})_{\alpha,\beta}\in \prod_{\alpha,\beta} G(U_{\alpha\beta})$ is a {\em \v{C}ech 1-cocycle} on $\mathcal{U}$ if $g_{\beta\gamma}g_{\alpha\beta}=g_{\alpha\gamma}$. This implies $g_{\alpha\alpha}=1$ and $g_{\alpha\beta}=g_{\beta\alpha}^{-1}$. Let $\check{H}^1(\mathcal{U},G)$ denote the set of all \v{C}ech 1-cocycles on $\mathcal{U}$, modulo the equivalence relation given by $(g_{\alpha\beta})_{\alpha,\beta}\sim(g'_{\alpha\beta})_{\alpha,\beta}$ if there exists $(h_{\alpha})_{\alpha}\in \prod_{\alpha} G(U_{\alpha})$ such that $g_{\alpha\beta}=h_{\beta}^{-1}g'_{\alpha\beta}h_{\alpha}$. We define the set
$$\check{H}^1(S,G):=\varinjlim_{\mathcal{U}}\check{H}^1(\mathcal{U},G)$$
to be the first \v{C}ech cohomology of $G$ on $S$. This is a pointed set in which the distinguished element is given by the cocycle $(1)$. Note that we have a canonical isomorphism of pointed sets $$\operatorname{Def}_{X_0}'(A)\simeq \check{H}^1(X_0,\underline{\operatorname{HS}}_k^A).$$ For $\overline{D}\in \check{H}^1(X_0,\underline{\operatorname{HS}}_k^A)$, we denote the corresponding deformation by $X_0[\overline{D}]\in \operatorname{Def}_{X_0}'(A)$. Recall that a sequence of pointed sets $$(P_1,p_1)\overset{f}{\to}(P_2,p_2)\overset{g}{\to}(P_3,p_3)$$ is called exact if $f(P_1)=g^{-1}(p_3)$. The exact sequences in Theorem \ref{ihs thm} induce exact sequences of pointed sets as follows:

\begin{proposition}\label{torsor seq}
 Let $X_0$ be an algebraic scheme over $k$ and
    $$\Phi:0\to k\cdot t\longrightarrow A'\overset{\varphi}{\longrightarrow} A\to 0$$
    be a small extension. Then we have the following two exact sequences of pointed sets
    $$
     1\to H^0(\underline{\operatorname{IHS}}^\Phi)\to H^0(\underline{\operatorname{HS}}_k^A)\to H^0(\operatorname{Ob}_{X_0}^\Phi) \overset{\partial_1} {\to} \check{H}^1(\underline{\operatorname{IHS}}^\Phi)\overset{\iota}{\to} \check{H}^1(\underline{\operatorname{HS}}_k^A)\overset{v}{\to}  \check{H}^1(\operatorname{Ob}_{X_0}^\Phi),$$
   $$ 0\to   H^0(\operatorname{Der}_k(\mathcal{O}_{X_0})) \to H^0(\underline{\operatorname{HS}}_k^{A'})\to H^0(\underline{\operatorname{IHS}}^\Phi)\to  $$
    $$\to \check{H}^1(\operatorname{Der}_k(\mathcal{O}_{X_0})) \to \check{H}^1(\underline{\operatorname{HS}}_k^{A'})\overset{\pi}{\to} \check{H}^1(\underline{\operatorname{IHS}}^\Phi)\overset{\partial_2}{\to}
     \check{H}^2(\operatorname{Der}_k(\mathcal{O}_{X_0})).$$
\end{proposition}
\begin{proof}
    This follows from \cite[Proposition 11.14]{gw} and \cite[Lemma 2.8]{BGL}. For example, $\partial_2$ is defined as follows: Let us take a \v{C}ech 1-cocycle $(E^{\alpha\beta})_{\alpha,\beta}\in\prod\underline{\operatorname{IHS}}^\Phi(U_{\alpha\beta})$ representing $\overline{E}\in\check{H}^1(\underline{\operatorname{IHS}}^\Phi)$. After refining the open covering if necessary, the $E^{\alpha\beta}$ lift to $D^{\alpha\beta}\in \underline{\operatorname{HS}}_k^{A'}(U_{\alpha\beta})$. Over $U_{\alpha\beta\gamma}$, $D^{\gamma\alpha}\circ D^{\beta\gamma}\circ D^{\alpha\beta}$ is of the form $d^{\alpha\beta\gamma}\langle t\rangle$ for $d\in H^0(U_{\alpha\beta\gamma},\operatorname{Der}_k(\mathcal{O}_{X_0}))$. Then $(d^{\alpha\beta\gamma})_{\alpha,\beta,\gamma}$ is a \v{C}ech $2$-cocycle. We define $\partial_2(\overline{E})\in \check{H}^2(\operatorname{Der}_k(\mathcal{O}_{X_0}))$ to be the class represented by $(d^{\alpha\beta\gamma})$.
\end{proof}

This shows that, given $X=X_0[\overline{D}]\in \operatorname{Def}'_{X_0}(A)$, the existence of a locally trivial lifting of $X$ to $A'$ is equivalent to non-emptiness of $(\pi\circ\iota)^{-1}(\overline{D})$.

\begin{proposition}\label{classes interpret prop} In the setting of Proposition \ref{torsor seq}, let $\overline{D}\in \check{H}^1(X_0,\underline{\operatorname{HS}}_k^A).$
\begin{enumerate}[label=(\roman*)]
\item $\iota^{-1}(\overline{D})\neq \emptyset$ if and only if $v(\overline{D})=0$.
\item $(\pi\circ\iota)^{-1}(\overline{D})\neq \emptyset$ if and only if $0\in \partial_2(\iota^{-1}(\overline{D}))$.

\end{enumerate}
\end{proposition}
\begin{proof}
    This is a consequence of exactness of the sequences in Proposition \ref{torsor seq}.
\end{proof}

Let us return to the situation that $G$ is a sheaf of groups on a topological space $S$, that $\mathcal{U}=\{U_\alpha\}_\alpha $ is an open covering of $S$ and that $(g_{\alpha\beta})_{\alpha,\beta}\in \prod_{\alpha,\beta} G(U_{\alpha\beta})$ is a \v{C}ech 1-cocycle representing $\overline{g}\in \check{H}^1(S,G)$. By $G[\overline{g}]$, we denote the twisted sheaf of groups obtained by gluing $G|_{U_\alpha}$ by $$G|_{U_{\alpha\beta}}\to G|_{U_{\beta\alpha}} , h\mapsto g_{\beta\alpha}hg_{\alpha\beta}^{-1}.$$
Now, in the setting of Proposition \ref{torsor seq}, if $\overline{D}\in \check{H}^1(X_0,\underline{\operatorname{HS}}_k^A)$, then the global section of $\underline{\operatorname{HS}}_k^A[\overline{D}]$ is identified with the infinitesimal deformation automorphism group of $X_0[\overline{D}]$ (that is, the group of $A$-automorphisms of $X_0[\overline{D}]$ which becomes the identity when restricted along $A\to k$). We have a natural group homomorphism
$$\operatorname{ob}_\Phi[\overline{D}]: \underline{\operatorname{HS}}_k^A[\overline{D}]\to \operatorname{Ob}_{X_0}^\Phi.$$

\begin{definition}\label{L(X) def}
    In the setting of Proposition \ref{torsor seq}, let $\overline{D}\in \check{H}^1(X_0,\underline{\operatorname{HS}}_k^A)$ and $X=X_0[\overline{D}].$ Define $L(X)\subset H^0(X_0,\operatorname{Ob}_{X_0}^\Phi)$ to be the image of $$H^0(X_0,\underline{\operatorname{HS}}_k^A[\overline{D}])\to H^0(X_0,\operatorname{Ob}_{X_0}^\Phi).$$
\end{definition}

\begin{proposition}\label{D action prop}
     In the setting of Proposition \ref{torsor seq}, let $\overline{D}\in \check{H}^1(X_0,\underline{\operatorname{HS}}_k^A)$ be such that $\iota^{-1}(\overline{D})\neq \emptyset$. Let $X=X_0[\overline{D}]$.
     \begin{enumerate}[label=(\roman*)]
\item We have a natural group action of $H^0(X_0,\operatorname{Ob}_{X_0}^\Phi)$ on the set $\iota^{-1}(\overline{D}).$
\item This action is transitive.
\item For $m\in H^0(X_0,\operatorname{Ob}_{X_0}^\Phi)$ and $\overline{E}\in \iota^{-1}(\overline{D})$, we have $m\cdot \overline{E}=\overline{E}$ if and only if $m\in L(X)$. In particular, we have an isomorphism of sets $H^0(\operatorname{Ob}_{X_0}^\Phi)/L(X)\simeq \iota^{-1}(\overline{D})$.
\end{enumerate}
\end{proposition}
\begin{proof}
    (i) Given $m\in H^0(\operatorname{Ob}_{X_0}^\Phi)$ and $\overline{E}\in  \iota^{-1}(\overline{D})$, we take a sufficiently fine open covering $\{U_\alpha\}$ of $X_0$ so that there exist $F^\alpha\in \underline{\operatorname{HS}}_k^A(U_\alpha)$ such that $\operatorname{ob}_\Phi(F^\alpha)=m|_{U_\alpha}$
    and that $\overline{E}$ is represented by $(E^{\alpha\beta})_{\alpha,\beta}\in \prod \underline{\operatorname{IHS}}^\Phi(U_{\alpha\beta})$. We define $m\cdot \overline{E}$ to be the class represented by $$((F^\beta)^{-1}\circ E^{\alpha\beta}\circ F^\alpha)_{\alpha,\beta}\in \prod \underline{\operatorname{IHS}}^\Phi(U_{\alpha\beta}).$$
    (ii) Let us take another $\overline{E}'\in \iota^{-1}(\overline{D})$ represented by $(E'^{\alpha\beta})_{\alpha,\beta}$. As $\overline{E}$ and $\overline{E}'$ define the same class of $\check{H}^1(X_0,\underline{\operatorname{HS}}_k^\Phi)$, there exist $F'^\alpha\in \underline{\operatorname{HS}}_k^A(U_\alpha)$ such that
    $$E'^{\alpha\beta}=(F'^{\beta})^{-1}\circ E^{\alpha\beta}\circ F'^{\alpha}.$$
    Then, the family $\{\operatorname{ob}_\Phi(F'^\alpha)\}_\alpha$ glues to yield $m'\in H^0(\operatorname{Ob}_{X_0}^\Phi)$ such that $m'\cdot \overline{E}=\overline{E}'.$
    (iii) In (i), if $m\cdot\overline{E}=\overline{E}$, there exist $E^\alpha\in  \underline{\operatorname{IHS}}^A(U_\alpha)$ such that
    $$(E^\beta)^{-1}\circ (F^\beta)^{-1}\circ E^{\alpha\beta}\circ F^\alpha\circ E^\alpha=E^{\alpha\beta}.$$
    Then, the family $\{F^\alpha\circ E^\alpha\}_\alpha\in \prod \underline{\operatorname{HS}}_k^A(U_\alpha)$ glues to yield a global section of $\underline{\operatorname{HS}}_k^A[\overline{D}]$ whose image in $H^0(\operatorname{Ob}_{X_0}^\Phi)$ is $m$. This implies $m\in L(X).$ Conversely, if $m\in L(X)$, then we can take $F^\alpha\in \underline{\operatorname{HS}}_k^A(U_\alpha)$ in such a way that $$F^\beta\circ E^{\alpha\beta}\circ (F^\alpha)^{-1}=E^{\alpha\beta}.$$
    This implies $m\cdot \overline{E}=\overline{E}$.
\end{proof}


\begin{proposition}\label{E action prop}
         In the setting of Proposition \ref{torsor seq}, let $\overline{E}\in \check{H}^1(X_0,\underline{\operatorname{IHS}}^\Phi)$ be such that $\pi^{-1}(\overline{E})\neq \emptyset$.
         \begin{enumerate}[label=(\roman*)]
             \item We have a natural group action of $\check{H}^1(X_0,\operatorname{Der}_k(\mathcal{O}_{X_0}))$ on $\pi^{-1}(\overline{E})$.
            \item This action is transitive.
         \end{enumerate}
\end{proposition}

\begin{proof}
    (i) Let $\overline{d}\in \check{H}^1(X_0,\operatorname{Der}_k(\mathcal{O}_{X_0}))$ and $\overline{D}\in \check{H}^1(X_0,\underline{\operatorname{HS}}_k^{A'})$ be represented by $(d^{\alpha\beta})$ and $(D^{\alpha\beta})$. Then, $\overline{d}\cdot \overline{D}$ is defined to be the class represented by the class $(D^{\alpha\beta}\circ d^{\alpha\beta}\langle t \rangle)$. This is indeed a group action because the image of
    $$\operatorname{Der}_k(\mathcal{O}_{X_0}) \overset{d\mapsto d\langle t\rangle}{\longrightarrow} \underline{\operatorname{HS}}_k^{A'}$$ lies in the center of $\underline{\operatorname{HS}}_k^{A'}.$ (ii) Suppose that $\overline{D}, \overline{D}'\in \pi^{-1}(\overline{E})$ are represented by $(D^{\alpha\beta}),(D'^{\alpha\beta})$.  This implies that there exist $E^\alpha\in \underline{\operatorname{IHS}}^\Phi(U_\alpha)$ such that $\varphi_*(D^{\alpha\beta})=E^\beta\circ\varphi_*(D'^{\alpha\beta})\circ (E^{\alpha})^{-1}$. As $E^\alpha$ lifts to some $D^\alpha\in \underline{\operatorname{HS}}_k^{A'}$ (after taking a refinement of the open covering), we have $D^{\alpha\beta}\equiv D^\beta\circ D'^{\alpha\beta} \circ (D^{\alpha})^{-1}$ modulo $t$. This implies that $(D^{\alpha\beta}\circ d^{\alpha\beta}\langle t\rangle)\sim (D'^{\alpha\beta})$ for some $d^{\alpha\beta}$ and we see that the family $(d^{\alpha\beta})_{\alpha,\beta}$ is a \v{C}ech 1-cocycle.
\end{proof}

\begin{remark}\label{fiber idetification rmk}
    Under the identifications $\operatorname{Def}_{X_0}'(k[\varepsilon]/\varepsilon^2)\simeq \check{H}^1(X_0,\operatorname{Der}_k(\mathcal{O}_{X_0}))$ and $\operatorname{Def}_{X_0}'(A')\simeq \check{H}^1(X_0,\underline{\operatorname{HS}}_k^{A'})$, the above action coincides with the natural one defined in terms of functor of artin rings (cf. Remark \ref{fiber action rmk}). For $\alpha: \operatorname{Def}_{X_0}'(A')\to \operatorname{Def}_{X_0}'(A)$ and $X=X_0[\overline{D}]\in\operatorname{Def}_{X_0}'(A)$, we have natural isomorphisms
    $$\alpha^{-1}(X)/\check{H}^1(X_0,\operatorname{Der}_k(\mathcal{O}_{X_0}))\simeq (\pi\circ \iota)^{-1}(\overline{D})/\check{H}^1(X_0,\operatorname{Der}_k(\mathcal{O}_{X_0}))\simeq \iota^{-1}(\overline{D}) \cap \partial_2^{-1}(0).$$
\end{remark}

\begin{definition}\label{nu def}
    In the setting of Proposition \ref{torsor seq}, let $\overline{D}\in \check{H}^1(X_0,\underline{\operatorname{HS}}_k^A)$ and $X=X_0[\overline{D}]$ be such that $\iota^{-1}(\overline{D})\neq \emptyset$. Take $\overline{E}\in \iota^{-1}(\overline{D})$.
    We define
    $$\nu_X^\Phi:  \ H^0(X_0,\operatorname{Ob}_{X_0}^{\Phi})\to \check{H}^2(X_0,\operatorname{Der}_k(\mathcal{O}_{X_0}))$$
    by
    $$\nu _X^\Phi(m)=\partial_2(m\cdot\overline{E}).$$
\end{definition}

Note that $\operatorname{ker} \nu_X /L(X)\simeq \iota^{-1}(D) \cap\partial_2^{-1}(0)$ by Proposition \ref{D action prop} (iii). Although the definition of $\nu_X$ depends on the choice of $\overline{E}$, the following proposition shows that $\operatorname{ker} \nu_X$ is unique up to linear transformations.

\begin{proposition}\label{k^p lin}      In the setting of Proposition \ref{nu def}, define
$$f^\Phi_{X_0}: \ H^0(X_0,\operatorname{Ob}_{X_0}^{\Phi})\to \check{H}^2(X_0,\operatorname{Der}_k(\mathcal{O}_{X_0}))$$
by
$$f^\Phi_{X_0}(m)=\partial_2(m\cdot\overline{E})-\partial_2(\overline{E}).$$
Then, if $X_0$ is a closed subscheme of a smooth algebraic scheme $P$, $f^\Phi_{X_0}$ depends on $\Phi$ and $X_0$, and not on $\overline{D}$ or $\overline{E}$. Moreover, $f^\Phi_{X_0}$ is a $k$-linear map in the sense that
$$f^\Phi_{X_0}(m+m')=f^\Phi_{X_0}(m)+f^\Phi_{X_0}(m'),$$ $$f^\Phi_{X_0}(a\cdot m)=a^{p^i}f^\Phi_{X_0}(m), \ a\in k,$$ where $i=i(\Phi)$ is as in Definition \ref{i Phi}.
\end{proposition}

\begin{proof}
    Let $m\in H^0(X_0,\operatorname{Ob}_{X_0}^{\Phi})$. Let $V_\alpha=\operatorname{Spec} T_\alpha$ be an sufficiently fine affine open cover of $P$ and let $U_\alpha=X_0\cap V_\alpha=\operatorname{Spec}R_\alpha$ be with $R_\alpha=T_\alpha/I_\alpha$. The class $\overline{E}$ is represented by $$E^{\alpha\beta}\in \underline{\operatorname{IHS}}^\Phi(U_{\alpha\beta})=\operatorname{IHS}^\Phi(R_{\alpha\beta}).$$ Then, $E^{\alpha\beta}$ lifts to $F^{\alpha\beta}\in \operatorname{HS}_k^{A'}(R_{\alpha\beta})$ and $F^{\alpha\beta}$ lifts to $G^{\alpha\beta}\in \operatorname{HS}_k^{A'}(T_{\alpha\beta})$. Similarly, $m|_{U_\alpha}$ lift to $F^{\alpha}\in \operatorname{HS}_k^A(R_\alpha)$ and $F^\alpha$ lift to $G^\alpha\in \operatorname{HS}_k^{A'}(T_\alpha)$, where $G^\alpha$ may not induce an element in $\operatorname{HS}_k^{A'}(R_\alpha)$. There exists an alternating family $d^{\alpha\beta}\in \operatorname{Der}_k(T_{\alpha\beta})$ such that $ d^{\alpha\beta}\langle t \rangle \circ (G^\beta)^{-1}\circ G^\alpha$ induces an element in $\operatorname{HS}_k^{A'}(R_{\alpha\beta})$. We see that $d^{\alpha\beta}+d^{\beta\gamma}+d^{\gamma\alpha}$ induces an element $\overline{d}^{\alpha\beta\gamma}\in \operatorname{Der}_k(R_{\alpha\beta\gamma})$ using Lemma \ref{bracket lem}. Now, $m\cdot \overline{E}$ is represented by $(F^\beta)^{-1}\circ D^{\alpha\beta}\circ F^\alpha$. By the next lemma, $ d^{\alpha\beta}\langle t \rangle \circ (G^\beta)^{-1}\circ G^{\alpha\beta}\circ G^\alpha$ induces an element in $\operatorname{HS}_k^{A'}(R_{\alpha\beta})$. Over $T_{\alpha\beta\gamma}$, we have
    $$d^{\gamma\alpha}\langle t \rangle\circ (G^\alpha)^{-1}\circ G^{\gamma\alpha}\circ G^\gamma\circ d^{\beta\gamma}\langle t \rangle\circ (G^\gamma)^{-1}\circ G^{\beta\gamma}\circ G^\beta \circ d^{\alpha\beta}\langle t \rangle\circ (G^\beta)^{-1}\circ G^{\alpha\beta}\circ G^\alpha$$
$$=(d^{\gamma\alpha}+d^{\beta\gamma}+d^{\alpha\beta})\langle t \rangle\circ (G^\alpha)^{-1} \circ G^{\gamma\alpha}\circ G^{\beta\gamma}\circ G^{\alpha\beta}\circ G^\alpha.$$
Here, note that there exists a $k$-derivation $d:T_{\alpha\beta\gamma}\to T_{\alpha\beta\gamma}$ such that  $G^{\gamma\alpha}\circ G^{\beta\gamma}\circ G^{\alpha\beta}\equiv d\langle t \rangle$ mod $I_{\alpha\beta\gamma}\otimes_k \mathfrak{m}_{A'}$. Thus, up to $I_{\alpha\beta\gamma}\otimes_k \mathfrak{m}_{A'}$, the above expression is congruent to 
$$\equiv(d^{\gamma\alpha}+d^{\beta\gamma}+d^{\alpha\beta})\langle t \rangle\circ (G^\alpha)^{-1} \circ G^\alpha\circ G^{\gamma\alpha}\circ G^{\beta\gamma}\circ G^{\alpha\beta}$$
$$=(d^{\gamma\alpha}+d^{\beta\gamma}+d^{\alpha\beta})\langle t \rangle\circ G^{\gamma\alpha}\circ G^{\beta\gamma}\circ G^{\alpha\beta}.$$
This implies that $f^\Phi_{X_0}$ is represented by the cocycle $(\overline{d}^{\alpha\beta\gamma})$, which depends on $\Phi$ and $X_0$ and not on $\overline{D}$ nor $\overline{E}$. This shows the first part of the proposition.

For $k$-linearlity, we can suppose that $\Phi$ is of the form $0\to k\cdot \varepsilon^q\to k[\varepsilon]/\varepsilon^{q+1}\to k[\varepsilon]/\varepsilon^q\to 0$ where $q=p^i$. For, there exists a diagram of the form
\[\begin{tikzcd}
	0 & {k\cdot \varepsilon^q} & {k[\varepsilon]/\varepsilon^{q+1}} & {k[\varepsilon]/\varepsilon^q} & 0 \\
	0 & {k\cdot t} & {A'} & A & 0
	\arrow[from=1-1, to=1-2]
	\arrow[from=1-2, to=1-3]
	\arrow["{\cdot 1}", from=1-2, to=2-2]
	\arrow[from=1-3, to=1-4]
	\arrow[from=1-3, to=2-3]
	\arrow[from=1-4, to=1-5]
	\arrow[from=1-4, to=2-4]
	\arrow[from=2-1, to=2-2]
	\arrow[from=2-2, to=2-3]
	\arrow["\varphi", from=2-3, to=2-4]
	\arrow[from=2-4, to=2-5]
\end{tikzcd}\]
that induces $\operatorname{Ob}^q_{X_0}= \operatorname{Ob}^\Phi_{X_0}$
and the above constructions are compatible with morphisms of small extensions. For $m'\in H^0(X_0,\operatorname{Ob}_{X_0}^{\Phi})$, suppose that there correspond $F'^\alpha\in \operatorname{HS}_k^A(R_\alpha), G'^\alpha\in \operatorname{HS}_k^{A'}(T_\alpha), d'^{\alpha\beta}\in \operatorname{Der}_k(T_{\alpha\beta})$.
Then, For $m+m'$, there correspond $F^\alpha\circ F'^\alpha,\ G^\alpha\circ G'^\alpha$ and $ d^{\alpha\beta}+d'^{\alpha\beta}.$
Here, note that 
$$(d^{\alpha\beta}+d'^{\alpha\beta})\langle \varepsilon^q \rangle\circ (G^\beta\circ  G'^\beta)^{-1}\circ G^\alpha\circ  G'^\alpha$$
$$=d'^{\alpha\beta}\langle \varepsilon^q \rangle\circ (G'^\beta)^{-1}\circ  (d^{\alpha\beta}\langle \varepsilon^q \rangle\circ (G^\beta)^{-1}\circ G^\alpha)\circ  G'^\alpha$$
induces an element in $\operatorname{HS}_k^q(R_{\alpha\beta})$ again by the following lemma. For $a\cdot m$, there correspond $a\bullet F^\alpha,a\bullet G^\alpha$ and $ a^qd^{\alpha\beta}$. This shows the the $k$-linearity of $f^\Phi_{X_0}$.
\end{proof}
\begin{lemma}Let $\Psi: 0\to k\cdot s\to A_2\to A_1\to 0$ be a small extension. Let $T$ be a formally smooth $k$-algebra, $I\subset T$ an ideal and $R:=T/I$. Let $D_1,D_2,D_3\in \operatorname{HS}_k^{A_1}(R)$. Let $E_1,E_2,E_3\in \operatorname{HS}_k^{A_2}(T)$ be their lifts and suppose that $E_2$ induces an element in $\operatorname{HS}_k^{A_2}(R)$, i.e., we have $E_2 (I \otimes_k A_2)\subset I\otimes_k A_2.$ Let $d\in \operatorname{Der}_k(T,T)$ be such that $d\langle s \rangle\circ E_1\circ E_3$ also induces an element in $\operatorname{HS}_k^{A_2}(R)$. Then, so does $d\langle s \rangle \circ E_1\circ E_2\circ E_3$.
\end{lemma}

\begin{proof}
    For every $i=1,2,3$, we have $E_i(I\otimes_k A_2)\subset I \otimes_k A_2 +T\otimes_k s$. Let $J:=I\otimes_k\mathfrak{m}_{A_2}$. For each $i$, $E_i$ induces $$\overline{E}_i:(T\otimes_kA_2)/J\to(T\otimes_kA_2)/J.$$ We can write
    $$\overline{E}_i(x\otimes 1)=x\otimes 1+\psi_i(x)\otimes s$$
    for $x\in I$ where $\psi_i: I\to T/I=R$. Also, $d\langle s \rangle$ induces $$\overline{d\langle s }\rangle:(T\otimes_kA_2)/J\to(T\otimes_kA_2)/J$$
    and we can write $$ \overline{d\langle s} \rangle(x\otimes 1)= x\otimes 1 + \psi(x)\otimes s,\ x\in I,\ \psi:I\to R.$$ By assumption, the restrictions of
    $\overline{E}_2$ and $\overline{d}\langle s \rangle\circ \overline{E}_1\circ \overline{E}_3$ to $I\otimes 1$ is the identity. On the other hand, for $x\in I$,
    $$\overline{E}_2(x\otimes 1)=x\otimes 1+\psi_2(x)\otimes s,$$
    $$\overline{d\langle s}\rangle\circ \overline{E}_1\circ \overline{E}_3(x\otimes 1)=x\otimes 1+(\psi+\psi_1+\psi_3)(x)\otimes s.$$
    This implies that $\psi_2=\psi+\psi_1+\psi_3=0.$ Thus, the restriction of $\overline{d}\langle s \rangle\circ \overline{E}_1\circ \overline{E}_2\circ \overline{E}_3$ to $I\otimes 1$ is also identity by the same reason. This means that $d\langle s \rangle \circ E_1\circ E_2\circ E_3$ induces an element in $\operatorname{HS}_k^{A_2}(R)$.
\end{proof}

Finally, let us summarize the above results in the language of locally trivial deformations.

\begin{theorem}\label{lt main thm}
    Let $X_0$ be an algebraic scheme over $k$,
    $$\Phi:0\to k\cdot t\longrightarrow A'\overset{\varphi}{\longrightarrow}A \to 0$$
    be a small extension and $X\in \operatorname{Def}'_{X_0}(A)$. Let $$\alpha: \operatorname{Def}'_{X_0}(A')\to \operatorname{Def}'_{X_0}(A).$$ Then:\begin{enumerate}[label=(\roman*)]
        \item There exists $$v_1(X)\in H^1(X_0,\operatorname{Ob}_{X_0}^\Phi)$$ such that, if $v_1(X)\neq 0$, then $\alpha^{-1}(X)=\emptyset$.
        \item Suppose $v_1(X)=0$. Then we have an isomorphism of sets
        $$\alpha^{-1}(X)/H^0(\operatorname{Der}_k(\mathcal{O}_{X_0}))\simeq \operatorname{ker}\nu_X^\Phi /\L(X),$$
        where $L(X), \nu_X^\Phi$ are as in Definition \ref{L(X) def}, \ref{nu def}.
        \item If $X_0$ is a closed subscheme of a smooth algebraic scheme, then $\operatorname{ker}\nu_X^\Phi$ is an affine sub-$k$-space of $H^0(X_0,\operatorname{Ob}_{X_0}^\Phi)$ that is unique up to linear transformations.
    \end{enumerate}
\end{theorem}

\begin{proof}
    Suppose that $X$ corresponds to $D\in \check{H}^1(\underline{\operatorname{HS}_k^A})$. We set $v_1(X)=v(D)$ where $v$ is as in Proposition \ref{torsor seq}. Then the results follow from the above propositions.
\end{proof}

\begin{corollary}\label{affine cor}
    If $X_0$ is affine, then any locally trivial deformation of $X_0$ is trivial.
\end{corollary}

\begin{proof}
    In the above theorem, note that $v_1(X)=0$, $H^0(\operatorname{Der}_k(\mathcal{O}_{X_0}))=0$ and $L(X)=\operatorname{ker} \nu_X^\Phi=H^0(X_0,\operatorname{Ob}_{X_0}^\Phi)$. This shows that $\alpha^{-1}(X)$ consists of a single point, i.e., that the lift of $X$ to $A'$ exists uniquely. Using this fact, we can show that the set $\operatorname{Def}'_{X_0}(A)$ consists only of the trivial deformation, by induction on the length of $A$.
\end{proof}
\section{A counterexample to $(H_1)$}

In this final section, we construct an algebraic curve in positive characteristic and verify that it does not satisfy Schlessinger's $(H_1)$ using Theorem \ref{lt main thm}.

Let $R=k[t^p, t^{p+1}]\subset k[t]$, $S=k[s^{2p+1},s^{2p+2}]\subset k[s]$, $U=\operatorname{Spec}R$, $V=\operatorname{Spec}S$. We glue $U$ and $V$ along $T:=k[t,t^{-1}]=k[s,s^{-1}]$ by $t=s^{-1}$, to obtain $$C_0=U\cup V.$$
Since the images of $R$ and $S$ in $T$ generate $T$, we see that $C_0$ is separated. In particular, $C_0$ is an algebraic curve.

We consider the following small extension
$$\Pi:0\to k\cdot \lambda^p\to k[\lambda]\overset{\pi}{\to} k[\varepsilon]\to 0$$
where the defining equations are $\varepsilon^p=0,\  \lambda^{p+1}=0$ and $\pi(\lambda)=\varepsilon$.

Let $C:=C_0\otimes_k k[\varepsilon]$ be the trivial deformation over $k[\varepsilon]$.
We would like to apply the results of the previous section in the case where $X_0=C_0$, $\Phi=\Pi$, $X=C$.

\begin{lemma}\label{inclusion proper lem}
  The inclusion $L(C)\subset H^0(C_0,\operatorname{Ob}_{C_0}^\Pi)$ is proper.
\end{lemma}

\begin{proof}
    By definition, $L(C)$ is the image of
    $$H^0(C_0,\underline{\operatorname{HS}}_k^{k[\varepsilon]})\to H^0(C_0,\operatorname{Ob}_{C_0}^\Pi)$$
    and it suffices to show that this map is not surjective.
    Note that $\operatorname{Ob}_{C_0}^\Pi$ is supported on a $0$-dimensional closed subset of $C_0$ since it vanishes on the smooth locus. Thus, it reduces to show that the composition
    $$H^0(C_0,\underline{\operatorname{HS}}_k^{k[\varepsilon]})\to H^0(C_0,\operatorname{Ob}_{C_0}^\Pi)\to H^0(U,\operatorname{Ob}_{C_0}^\Pi)$$
    is not surjective. In fact, we show that the composition is zero and that $H^0(U,\operatorname{Ob}_{C_0}^\Pi)$ is nonzero.
    
     Let $\sigma\in H^0(C_0,\underline{\operatorname{HS}}_k^{k[\varepsilon]})$. We have the following commutative diagram
\[\begin{tikzcd}
	{H^0(C_0,\underline{\operatorname{HS}}_k^{k[\varepsilon]}) } & {H^0(C_0,\operatorname{Ob}_{C_0}^\Pi)} \\
	{\operatorname{HS}_k^{k[\varepsilon]}(R)=H^0(U,\underline{\operatorname{HS}}_k^{k[\varepsilon]}) } & {H^0(U,\operatorname{Ob}_{C_0}^\Pi)}
	\arrow[from=1-1, to=1-2]
	\arrow[from=1-1, to=2-1]
	\arrow[from=1-2, to=2-2]
	\arrow[from=2-1, to=2-2]
\end{tikzcd}\]
By Lemma \ref{curve coefficient lemma} below, the restriction $\sigma_U\in\operatorname{HS}_k^{k[\varepsilon]}(R) $ of $\sigma$ to $U$ satisfies
$$\sigma_U(t^p)=t^p,$$
$$\sigma_U(t^{p+1})=t^{p+1}+\varepsilon t^p f_1+ \varepsilon t^{p+1} f_2+ \varepsilon^2 f_3, \ f_1,f_2,f_3\in R\otimes_k k[\varepsilon].$$
Let $g_1,g_2,g_3 \in R\otimes_k k[\lambda]$ be lifts of $f_i$. Then, $\sigma_U$ can be lifted to $\widetilde{\sigma}_U\in\operatorname{HS}_k^{k[\lambda]}(R)$ defined by
$$\widetilde{\sigma}_U(t^p)=t^p+\lambda^p(g_1^p+t^pg_2^p),$$
$$\widetilde{\sigma}_U(t^{p+1})=t^{p+1}+\lambda t^p g_1+ \lambda t^{p+1} g_2+ \lambda^2 g_3,$$
which is well-defined because
$$\widetilde{\sigma}_U((t^p)^{p+1}-(t^{p+1})^p)=(t^{p^2+p}+t^{p^2}\lambda^p(g_1^p+t^pg_2^p))-(t^{p^2+p}+\lambda^pt^{p^2}g_1^p+\lambda^pt^{p^2+p}g_2^p)=0.$$
This shows that $\operatorname{ob}_{\Pi}(\sigma_U)=0$ and the composition $$H^0(C_0,\underline{\operatorname{HS}}_k^{k[\varepsilon]})\to H^0(C_0,\operatorname{Ob}_{C_0}^\Pi)\to H^0(U,\operatorname{Ob}_{C_0}^\Pi)$$ is the zero map. 

Next, let us consider $D\in\operatorname{HS}_k^{k[\varepsilon]}(R) $ defined by
$$D(t^p)=t^p,$$
$$D(t^{p+1})=t^{p+1}+\varepsilon.$$
This is well-defined because $D((t^p)^{p+1}-(t^{p+1})^p)=0$. We would like to show that there is no $E\in\operatorname{HS}_k^{k[\lambda]}(R)$ that restricts to $D$. For such $E$, we would have
$$E(t^p)=t^p+\lambda^p f_4,$$
$$E(t^{p+1})=t^{p+1}+\lambda+\lambda^pf_5, \ f_4,f_5\in R.$$
Then we have
$$0=E((t^p)^{p+1}-(t^{p+1})^p)=(t^{p^2+p}+\lambda^pt^pf_4)-(t^{p^2+p}+\lambda^p)=\lambda^p(-1+t^pf_4)\neq 0,$$
hence contradiction. This shows that $\operatorname{ob}_{\Pi}(D)\neq 0$ and, in particular, we have $H^0(U,\operatorname{Ob}_{C_0}^\Pi)\neq 0.$ \end{proof}

\begin{lemma}\label{curve coefficient lemma} For $\sigma\in H^0(C_0,\underline{\operatorname{HS}}_k^{k[\varepsilon]})$ and $\sigma_U$ is the restriction of $\sigma$ to $U$, then we have 
$$\sigma_U(t^p)=t^p,$$
$$\sigma_U(t^{p+1})-t^{p+1}\in (\varepsilon t^p, \varepsilon t^{p+1}, \varepsilon^2)\subset k[t^p,t^{p+1}]\otimes_k k[\varepsilon].$$

\begin{proof}
We can identify the sections of $\mathcal{O}_C=\mathcal{O}_{C_0}\otimes_k k[\varepsilon]$ over non-empty open subsets with their image in the stalk of $\mathcal{O}_C$ at the generic point $\eta$. In particular, $\sigma(t)$ is of the form $t+ x\varepsilon$, $x\in \mathcal{O}_{C_0,\eta}\otimes_k k[\varepsilon]$ so that $\sigma(t^p)=(t+ x\varepsilon)^p=t^p$ because $p=\varepsilon^p=0$. On the other hand, as $\sigma_U(t^{p+1})\in H^0(U,\mathcal{O}_C)$, it is of the form $$\sigma(t^{p+1})=t^{p+1}+(a_0+a_pt^p+a_{p+1}t^{p+1}+\mathrm{higher \ order \ terms})$$
where the coefficients $a_i\in \varepsilon\cdot k[\varepsilon]$ are uniquely determined. We obtain $$\sigma(t)=\sigma(t^{p+1}/t^p)=t+(a_0t^{-p}+a_p+a_{p+1}t+\mathrm{higher \ order \ terms}).$$ We can replace $t$ and $ U$ by $s$ and $ V$ in the above arguments and we obtain $$\sigma(s^{2p})=s^{2p},$$ $$\sigma(s^{2p+1})=s^{2p+1}+(b_0+b_{2p+1}s^{2p+1}+b_{2p+2}s^{2p+2}+b_{4p+2}s^{4p+2}+\mathrm{higher \ order \ terms})$$
where $b_i\in \varepsilon\cdot k[\varepsilon]$. The inverse $\sigma(s^{2p+1})^{-1}$ is congruent to
    $$\frac{1}{s^{2p+1}}-\frac{1}{s^{4p+2}}(b_0+b_{2p+1}s^{2p+1}+b_{2p+2}s^{2p+2}+b_{4p+2}s^{4p+2}+\mathrm{higher \ order \ terms})$$
modulo $\varepsilon^2$. We can combine these equations to obtain
$$\begin{aligned} \sigma(t)&=\sigma(s^{-1})=\sigma(\frac{s^{2p}}{s^{2p+1}}) \\ &\equiv \frac{s^{2p}}{s^{2p+1}}-\frac{s^{2p}}{s^{4p+2}}(b_0+b_{2p+1}s^{2p+1}+b_{2p+2}s^{2p+2}+b_{4p+2}s^{4p+2}+\mathrm{higher \ order \ terms}) \\ &=t+(b_0t^{2p+2}+b_{2p+1}t+b_{2p+2}+b_{4p+2}t^{-2p}+\mathrm{lower \ order \ terms}). \\ \end{aligned}$$
By comparing coefficients modulo $\varepsilon^2$, we see that $a_0\in \varepsilon^2\cdot k[\varepsilon]$ and
$$\sigma_U(t^{p+1})-t^{p+1}\in (\varepsilon t^p, \varepsilon t^{p+1}, \varepsilon^2).$$\end{proof}
\end{lemma}

\begin{proposition}
    Let $\beta: \operatorname{Def}'_{C_0}(k[\lambda])\to \operatorname{Def}'_{C_0}(k[\varepsilon]).$ Then, the set $$\beta^{-1}(C)/H^0(C_0,\operatorname{Der}_k(\mathcal{O}_{C_0}))$$ is neither empty nor $\{\mathrm{pt}\}$. In particular, the functor $\operatorname{Def}'_{C_0}$ does not satisfy Schlessinger's condition $(H_1)$.
\end{proposition}
\begin{proof}
    We apply Theorem \ref{lt main thm}. Note that $H^2(C_0,\operatorname{Der}_k(\mathcal{O}_{C_0}))$ and $\operatorname{ker}\nu_C^\Pi=H^0(C_0,\operatorname{Ob}_{C_0}^\Phi)$. This implies that
    $$\beta^{-1}(C)/H^0(C_0,\operatorname{Der}_k(\mathcal{O}_{C_0}))\simeq \operatorname{ker}\nu_C^\Pi/L(C)$$
    consists of more than one elements. The latter part follows from the fact that this set would be either empty or a single point if $\operatorname{Def}'_{C_0}$ satisfied $(H_1)$ (cf. the proof of \cite[Theorem 2.11]{Sch}).
\end{proof}

\begin{theorem}\label{counterexample theorem}
    For algebraically closed base field of any positive characteristic, there exists an algebraic curve $C_0$ whose locally trivial deformation functor $\operatorname{Def}'_{C_0}$ does not satisfy Schlessinger's condition $(H_1)$.
\end{theorem}

\bibliographystyle{plain}
\bibliography{hasse}
\end{document}